\documentclass[11pt]{amsart}
\usepackage{amssymb,enumerate}
\usepackage{amsmath}
\usepackage{bbm}           
\usepackage{bm}
\usepackage{eso-pic,graphicx}
\usepackage{tikz}
\usepackage{esint}
\usepackage{ulem}

\usepackage[colorlinks=true, pdfstartview=FitV, linkcolor=blue, citecolor=blue, urlcolor=blue]{hyperref}

\usepackage{mathrsfs}

\newtheorem{theorem}{Theorem}[section]

\newtheorem{corollary}[theorem]{Corollary}

\newtheorem{lemma}[theorem]{Lemma}

\newtheorem{remark}[theorem]{Remark}

\def\rr{{\mathbb R}}
\def\R{{\mathbb R}}
\def\zz{{\mathbb Z}}
\def\nn{{\mathbb N}}
\def\C{{\mathbb C}}

\def\w{\widehat}

\def\M{\mathbf{M}}
\def\I{\mathbf{I}}
\def\T{\mathbf{T}}
\def\K{\mathbf{K}}
\def\k{\mathbf{k}}

\def\({\left(}
\def\){\right)}
\def\[{\left[}
\def\]{\right]}
\def\<{\left<}
\def\>{\right>}

\def\R{\mathbb R}
\def\C{\mathbb C}
\def\Z{\mathbb Z}
\def\N{\mathbb N}
\def\S{\mathscr S}

\def\supp{\text{supp}}
\def\less{\lesssim}

\DeclareSymbolFont{bbold}{U}{bbold}{m}{n}
\DeclareSymbolFontAlphabet{\mathbbn}{bbold}
\allowdisplaybreaks

\def\Q{\mathcal{Q}}
\def\p{\widetilde{p}}
\def\q{\widetilde{q}}
\def\lpq{L^pL^q}

\begin{document}

\title[Smoothing bilinear operators and  Leibniz-type rules]
{Smoothing properties of bilinear operators and Leibniz-type  rules in Lebesgue and  mixed Lebesgue spaces}

\bigskip

\author[J. Hart, R.H. Torres, X. Wu]
{Jarod Hart, Rodolfo H. Torres, Xinfeng Wu}

\address{Higuchi Biosciences Center\\
University of Kansas\\
 Lawrence, KS 66047}
\email{jvhart@ku.edu}

\address{Department of Mathematics\\
University of Kansas\\
 Lawrence, KS 66045-7523}
\email{torres@ku.edu}

\address{Department of Mathematics\\
University of Kansas\\
 Lawrence, KS 66045-7523}
\email{wuxf@ku.edu}

\thanks{}

\subjclass[2010]{Primary: 42B20; Secondary: 42B15, 47G99}
\keywords{Bilinear operators, multipliers, maximal function, smoothing properties, fractional derivatives, Leibniz rule, mixed Lebesgue spaces}

\begin{abstract}
We prove that bilinear fractional integral operators and similar multipliers are smoothing in the sense that
they improve the regularity of functions. We also treat bilinear singular multiplier operators which preserve regularity and obtain several Leibniz-type rules in the contexts of Lebesgue and mixed Lebesgue spaces.
\end{abstract}

\maketitle

%
\section{Introduction}

Let $K_\nu$ be an integral operator  of order $-\nu$. That is, let $K_\nu$ be of the form
\begin{equation}\label{nu}
K_\nu f(x) = \int_{\R^n} k_\nu (x,y) f(y)\,dy,
\end{equation}
where the kernel satisfies the estimate
$
|k_\nu (x,y) | \lesssim \frac {1}{|x-y|^{n-\nu}},
$
for some $0<\nu<n$. It is easy to see that $K_\nu$ is {\it smoothing}, or rather {\it improving}, in the scale of Lebesgue spaces, in the sense that it maps a Lebesgue space into another one with a larger exponent. More precisely,
$$
K_\nu:L^p \to L^q
$$
provided $0<1/q=1/p - \nu/n<1$. Under suitable additional regularity and cancellation conditions  (see e.g. \cite{T}), such $K_\nu$ is also smoothing in the Sobolev scale. Namely
$$
K_\nu:L^p \to \dot W^{\nu,p}
$$
where $\dot W^{\nu,p}$ is the homogeneous Sobolev space of functions with their derivative of order $\nu$ in $L^p$ (the precise definitions of all function spaces used in this article are given in Section 2 below). This is a stronger smoothing property, since by Sobolev embedding $\dot W^{\nu,p} \subset L^q$, when $p$ and $q$ are related as above. Of course, the most classical situation is that of the Riesz potential operators
$$
I_\nu f(x) = c_\nu \int_{\R^n} \frac {1}{|x-y|^{n-\nu}}f(y)\,dy,
$$
where the constant $c_\nu$ is selected so that the Fourier transform of $I_\nu f$ is given by $$\widehat{I_\nu f}(\xi) = |\xi|^{-\nu}\widehat f(\xi).$$

It is immediate that by defining $\widehat {D^sf}(\xi) = |\xi|^{s}\widehat f(\xi)$ we have  for $s<\nu$,
\begin{equation} \label{calculus}
D^s I_\nu = I_{\nu-s}.
\end{equation}

Formally, the case $\nu=0$ in \eqref{nu} corresponds to Calder\'on-Zygmund operators which are no longer smoothing, but a slight modification of this simple calculus in \eqref{calculus} still holds for convolution operators. For example for $n>1$,
\begin{equation}\label{cal2}
\partial_j I_1f=R_jf
\end{equation}
where for  $j=1,\ldots,n$,  $R_j$ are the Riesz transforms in $\R^n$ given by the multiplier  $\widehat{R_jf}(\xi)= -i \xi_j |\xi|^{-1}\w f (\xi)$. As operators {\it of order zero}, the Riesz transforms $R_j$ are not smoothing, but since they commute with derivative,
\begin{equation}\label{cal3}
 D^s (R_jf)= R_j(D^sf),
 \end{equation}
they preserve both Lebesgue and Sobolev spaces for $1<p<\infty$.

Properly interpreted the calculus in  \eqref{calculus}--\eqref{cal3} extends not only to other multiplier operators, but also beyond the convolution case to several classes of pesodifferential operators and even more general non-convolution operators of Calder\'on-Zygmund type (see e.g the book by Stein \cite{S} for several results and references to the vast literature in the subject).

In this article, we are interested in stating and proving analogous versions of \eqref{calculus}--\eqref{cal3} for bilinear multiplier operators, improving and extending numerous results already in the literature in the subject and uncovering several completely new ones. The prototypes for our work for
$0<\nu<2n$  will be bilinear fractional  integral operators, while for $\nu=0$ they will be Coifman-Meyer multipliers. We will obtain, however,  results for more general operators under minimal regularity assumptions on the multiplier which do not allow for pointwise smooth estimates on their corresponding kernels.

The bilinear fractional integral operators are defined for $0<\nu<2n$ by
$$
\I_\nu(f,g)(x)=C_\nu \int_{\R^{2n}}\frac{1}{(|x-y|^2+|x-z|^2)^{(2n-\nu)/2}}f(y)g(z)\,dy dz.
$$
The constant $C_\nu$ is chosen again so that, using the Fourier transform, we have the representation
$$
{\I_\nu(f,g)}(x) = \int_{\R^{2n}}\frac{1}{(|\xi|^2+|\eta|^2)^{\nu/2}}e^{-ix(\xi+\eta)}\w f(\xi)\w g(\eta)\,d\xi d\eta.
$$
More generally we can consider for  $0\leq\nu<2n$ bilinear multipliers of the form
$$
\T_{m_\nu}(f,g)(x) = \int_{\R^{2n}}m_\nu(\xi,\eta) e^{-ix(\xi+\eta)}\w f(\xi)\w g(\eta)\,d\xi d\eta,
$$
where
\begin{equation}\label{verysmooth}
|\partial^\beta_\xi\partial^{\gamma}_\eta  m_\nu(\xi,\eta)| \lesssim_{\beta\gamma} (|\xi|+|\eta|)^{-\nu-|\beta|-|\gamma|}.
\end{equation}
Note that we are allowing now $\nu=0$, which corresponds to the case of the nowadays classical Coifman-Meyer multipliers.  We will actually  treat multipliers where the pointwise regularity estimates in \eqref{verysmooth} are replaced by H\"ormander-type ones using only appropriate Sobolev space regularity.

Roughly speaking, if $\T_\nu$ is a bilinear  operator of order $-\nu$ described above, we will show that $\T_\nu(f,g)$ {\it has $\nu$ more derivatives than $f$ and $g$} (hence it is smoothing if $\nu>0$).  Our main results could be interpreted by saying that 
\begin{equation}\label{cal8}
D^s\T_\nu(f,g) \sim \T_{0}(D^{s-\nu} f,g) + \T_{0}(f,D^{s-\nu} g),
\end{equation}
where $T_{0}$ is an operator of order $0$.

In making these informal statements precise, we need to review some of the existing literature alluded to before. Our recount is not intended to be exhaustive, but we shall rather point out some of the results most closely related to ours.
As it will be clear from our narrative below, there is a high level of interest in the subject and a very active community working on similar problems. Several overlapping recent results have been obtained independently by different authors.

As already mentioned, for $\nu=0$ the operators in \eqref{verysmooth} are Coifman-Meyer multipliers as studied by those authors in \cite{C-M1}-\cite{C-M4}. They are examples of operators within the multilinear Calder\'on-Zygmund theory further developed by Christ-Journ\'e \cite{CJ}, Kenig-Stein \cite{KS} and Grafakos-Torres \cite{GT}. In particular, bilinear Calder\'on-Zygmund operators are operators of the form
\begin{equation}\label{kernelones}
\K (f,g)(x) = \int_{\R^{2n}} \k (x,y,z) f(y)g(z)\,dydz
\end{equation}
for $x\notin \supp f \cap \supp g$, where the kernel satisfies the estimates
\begin{equation}\label{kernelestimates}
|D^\alpha_x D^\beta_y D^\gamma_z \k(x,y,z)|\less (|x-y|+|x-z|)^{-2n-|\alpha|-|\beta|-|\gamma|},
\end{equation}
and such that they act as the product of functions on Lebesgue spaces, i.e,
$$
\K : L^{p_1}\times L^{p_2} \to L^q
 $$
 for $1<p_1,p_2<\infty$, $\frac{1}{p_1}+\frac{1}{p_2}=\frac{1}{q}$ (appropriate end-point results hold too).
Other examples of these operators are provided by bilinear pesuodifferential operators of order zero. For $m \in \R$, a bilinear pseudodifferential operator of order $m$ is given by
$$
P_{a_m}(f,g)(x) =  \int_{\R^{2n}}a_m(x,\xi,\eta) e^{-ix(\xi+\eta)}\w f(\xi)\w g(\eta)\,d\xi d\eta,
$$
where
\begin{equation}\label{verysmoothsymbol}
|\partial^\rho_x\partial^\beta_\xi\partial^{\gamma}_\eta  a_m(x,\xi,\eta)| \lesssim (1+|\xi|+|\eta|)^{m-|\beta|-|\gamma|}.
\end{equation}
B\'enyi-Torres \cite{BT} showed that for $m =0$ these bilinear Calder\'on-Zygmund operators also satisfy for $s>0$
and $1/p_1 +1/p_2 =1/q <1$ 
the estimate
$$
\|J^sP_{a_0}(f,g)\|_{L^q} \lesssim \|J^s f\|_{L^{p_1}}\| g\|_{L^{p_2}} + \| f\|_{L^{p_1} }\|J^s g\|_{L^{p_2}},
$$
 where $J^s$ is the {\it inhomogeneous derivative operator} $$\w {(J^sf)}(\xi) = (1+|\xi|^2)^{s/2}\w f(\xi).$$ Moreover, along the lines of \eqref{cal8}, B\'enyi-Nahmod-Torres \cite{BNT} showed that for a symbol of order $m>0$,
\begin{equation}\label{cal9}
P_{a_m}(f,g)= P_{b_0}(J^mf,g) + P_{c_0}(f,J^mg),
\end{equation}
for some symbols of order zero $b_0$ and $c_0$, which gives then
$$
\|P_{a_m}(f,g)\|_{L^q} \lesssim \|J^m f\|_{L^{p_1}}\| g\|_{L^{p_2}} + \| f\|_{L^{p_1} }\|J^m g\|_{L^{p_2}},
$$
for all $1<p_1, p_2<\infty$ and $\frac{1}{p_1}+\frac{1}{p_2}=\frac{1}{q}$. This idea goes back to the work of Kato-Ponce \cite{KP}.
Similar estimates for more general classes of symbols were given by B\'enyi et al \cite{BMNT} and \cite{BBMNT} and Naibo \cite{N}. Several classes of operators in the H\"ormander classes $BS^m_{\rho,\delta}$ given by symbols satisfying the differential inequalities
\begin{equation}\label{generalsymbols}
|\partial^\alpha_x\partial^\beta_\xi\partial^{\gamma}_\eta  a_m(x,\xi,\eta)| \lesssim (1+|\xi|+|\eta|)^{m+\rho|\alpha|-\delta(|\beta|+|\gamma|)},
\end{equation}
for $0\leq \rho, \delta\leq 1$ were considered in those works. In particular, it was shown in \cite{N} that the boundedness
$L^{p_1} \times L^{p_1} \to L^{p}$ with $1/p_1+1/p_2 =1/p$, $1<p_1,p_2<\infty$ of an operator with symbol in a class $BS^m_{\rho,\delta}$, automatically implies its boundedness on Besov spaces with positive smoothness and based on the same $L^p$ exponents. It was also proved in \cite{N} that the same result is true for any bilinear multiplier operator mapping $L^{p_1} \times L^{p_1} \to L^{p}$.  A similar result for multipliers was obtained in \cite{BNT} in the scale of Sobolev spaces but with $p>1$.

The boundedness properties of the operators  $ \I_\nu$ in the scale of Lebesgue spaces were studied by  Kenig-Stein \cite{KS}. They showed that
\begin{equation}\label{basicfrac}
 \I_\nu:
 L^{p_1}\times L^{p_2} \to L^q
 \end{equation}
 for $1<p_1,p_2<\infty$, $0<\frac{1}{p_1}+\frac{1}{p_2}-\frac{\nu}{n}=\frac{1}{q}$, and $0<\nu<2n$.  Bernicot et al \cite{BMMN} looked at bilinear pseudodifferential  operators $P_{a_m}$ with $m<0$ and also homogeneous version $\dot P_{a_m}$, where the estimates in \eqref{verysmoothsymbol} are modified by replacing $(1+|\xi|+|\eta|)$ with $(|\xi|+|\eta|)$.  In particular, the operators $\I_\nu$ (or more generally $\T_{m_\nu}$ satisfying \eqref{verysmooth}) are homogeneous bilinear pseudodifferential operators of order $m=-\nu$. The authors in \cite{BMMN} showed, using a calculus similar to \eqref{cal9}, that
\begin{equation}\label{lpsmooth}
\|\I_{\nu}(f,g)\|_{L^q} \lesssim \|D^{s-\nu}f\|_{L^{p_1}}\| g\|_{L^{p_2}} + \| f\|_{L^{p_1} }\|D^{s-\nu} g\|_{L^{p_2}}
\end{equation}
 if $1<p_1,p_2<\infty$, $0<\frac{1}{p_1}+\frac{1}{p_2}-\frac{s}{n}=\frac{1}{q}$, and $0<s<2n$ and $\nu \leq s$.  We will show that actually
 \begin{equation}\label{sobolevsmooth}
\|D^s\I_{\nu}(f,g)\|_{L^p} \lesssim \|D^{s-\nu}f\|_{L^{p_1}}\| g\|_{L^{p_2}} + \| f\|_{L^{p_1} }\|D^{s-\nu} g\|_{L^{p_2}},
\end{equation}
 if $1<p_1,p_2<\infty$, $\frac{1}{p_1}+\frac{1}{p_2}=\frac{1}{p}$, $0<\nu<2n$ and $s>\max(0,\frac{n}{p}-n)$. This is now a smoothing property on the Sobolev scale and by Sobolev embedding an improvement of \eqref{lpsmooth} for some range of the exponents. In particular,
 \begin{equation}\label{basicsobolev}
 \I_\nu:
 L^{p_1}\times L^{p_2} \to \dot W^{\nu,p}
 \end{equation}
 for $1<p_1,p_2<\infty$, $0<\frac{1}{p_1}+\frac{1}{p_2}=\frac{1}{p}$, and $0<\nu<2n$, improving \eqref{basicfrac}.

We point out that other smoothing-type estimates have been proved for the bilinear fractional integral operators before.  For example, in \cite{AHIV}, Aimar et al proved that the $\I_\nu$ maps from products of Lebesgue spaces with appropriate indices into certain Campanato-$BMO$ type spaces when $\frac{1}{p_1}+\frac{1}{p_2}\leq \frac{\nu}{n}$.   Such spaces provide the right setting when working on spaces of homogeneous type.
More recently, Chaffee-Hart-Oliveira \cite{CJO} showed using different methods that
\begin{equation}\label{cjo}
\I_\nu :L^{p_1}\times L^{p_2} \to I_s(BMO),
\end{equation}
for certain ranges of $1<p_1,p_2<\infty$ and $0\leq s<\nu$ satisfying $\frac{1}{p_1}+\frac{1}{p_2}=\frac{\nu-s}{n}$. Note that \eqref{basicsobolev} is also an improvement of \eqref{cjo} whenever $\nu-s<n$ since $\dot W^{\nu,p} \subset I_s(BMO)$  if  $\frac{1}{p}=\frac{\nu-s}{n}<1$. The results in \cite{CJO}, however, apply to a larger range of exponents and also to more general operators that we cannot cover with our techniques.

The estimate  \eqref{sobolevsmooth}, and hence \eqref{basicsobolev}, hold for the multipliers $\T_{m_\nu}$ as well,
 \begin{equation}\label{sobolevsmoothT}
\|D^s{\T_{m_\nu}}(f,g)\|_{L^p} \lesssim \|D^{s-\nu}f\|_{L^{p_1}}\| g\|_{L^{p_2}} + \| f\|_{L^{p_1} }\|D^{s-\nu} g\|_{L^{p_2}},
\end{equation}
and we can allow the Coifman-Meyer case $\nu=0$ too. We note that after our work was completed we received an independent preprint from Brummer-Naibo \cite{BN} dealing with homogeneous pseudodifferential operators of different orders. Their results can be applied to smooth multipliers too, obtaining estimates similar to \eqref{sobolevsmoothT}. The techniques employed by these authors, however, are very different from ours. They rely on smooth molecular decompositions. To some extent, they are a bilinear counterpart of the results by Torres \cite{T1} and Grafakos-Torres \cite{GT0} in the linear case.  The results in \cite{BN} apply also to $x$-dependent smooth symbols, which cannot be treated by our methods, but the multipliers we study have very limited amount of regularity and, as far as we know, estimates involving smooth molecular decompositions require pointwise smoothness on the symbols.

Taking $m_0=1$, \eqref{sobolevsmoothT}   leads to the already known Leibniz rule
 \begin{equation}\label{classical}
\|D^s(f g)\|_{L^p} \lesssim \|D^{s}f\|_{L^{p_1}}\| g\|_{L^{p_2}} + \| f\|_{L^{p_1} }\|D^{s} g\|_{L^{p_2}}.
\end{equation}
This estimate also has a long history starting with works of Kato-Ponce \cite{KP} and Christ-Weinstein \cite{CW}. The validity of the rule for the optimal range of exponents  $1/2 < p< \infty$, $1<p_1,p_2\leq\infty$, $\frac{1}{p}=\frac{1}{p_1}+\frac{1}{p_2}$ and $s>\max (0, \frac{n}{p}-n)$
or $s$ a positive even integer, was finally settled by Muscalu-Schlag \cite{MS} and Grafakos-Oh \cite{GO}. We refer to \cite{GO} for previous works, additional weak-type estimates, and counterexamples for the limitations on $s$. The case $p_1=p_2=p= \infty$ was then further considered by Grafakos-Maldonado-Naibo \cite{GMN} and completely resolved by Bourgain-Li \cite{BL}.

Motivated by applications in time-dependent partial differential equations, there has also been some interest in obtaining Leibniz rules in the mixed Lebesgue spaces $L^p_tL^q_x(\rr^{n+1})$. The first such result involved a commutator estimates with fractional derivatives only in the space variable $x$ and was obtain by Kenig-Ponce-Vega \cite{KPV}. Torres-Ward \cite{TW} obtain then a result with the full derivatives in all variables. Denoting by $D_{t,x}^s$ the fractional derivatives in $\rr^{n+1}$, it was a shown in \cite{TW} that
\begin{equation}\label{primera}
\|D_{t,x}^s(fg)\|_{L^pL^q} \lesssim \|f\|_{L^{p_1}L^{q_1}}\|D_{t,x}^sg\|_{L^{p_2}L^{q_2}} + \|D_{t,x}^sf\|_{L^{p_1}L^{q_1}} \|g\|_{L^{p_2}L^{q_2}},
\end{equation}
for  $1<p,q,p_1,q_1, p_2,q_2<\infty$, $\frac{1}{p}=\frac{1}{p_1}+\frac{1}{p_2}$, $\frac{1}{q}=\frac{1}{q_1}+\frac{1}{q_2}$, and $s>0$. Notice that this restricts the target indices $p,q$ to be larger than $1$.  In this article, we adapt  the arguments in \cite{GO} to mixed Lebesgue spaces and obtain
\begin{align}\label{multipliermixlebesgue}
 & \|D_{t,x}^sT_{m_\nu}(f,g)\|_{L^{p,q}} \nonumber \\
&\hspace {2cm}  \lesssim  \|D_{t,x}^{s-\mu} f\|_{L^{p_1}L^{q_1}}\|g\|_{L^{p_2}L^{q_2} }
  +\|f\|_{L^{p_1}L^{q_1}}\|D_{t,x}^{s-\mu}g\|_{L^{p_2}L^{q_2}}
\end{align}
for  $0\le\nu<2n+2$, $1<p_i,q_i<\infty$, $i=1,2$, $1/p=1/p_1+1/p_2$, $1/q=1/q_1+1/q_2$,
and $s\in 2\nn$ or $s>\max(0,\frac{n+1}{p}-(n+1),\frac{n+1}{q}-(n+1))$.  In particular the case $\nu=0$ in \eqref{multipliermixlebesgue} can be used to extend \eqref{primera} to the full range $1/2<p,q<\infty$ for the appropriate values of $s$.

We mention that other authors have considered mixed derivatives variations of \eqref{primera} too. When $n=1$, let $D_x^s$ and $D_t^s$ be the fractional derivatives in the respective one-dimensional variables $x$ and $t$. Benea-Muscalu \cite{BM} showed first that in $\rr^{1+1}$,
\begin{align}
 \|\, D_t^\beta D_x^\alpha & (fg)  \|_{L^pL^q}        \nonumber \\
  & \lesssim   \|f\|_{L^{p_1}L^{q_1}} \|\,D_t^\beta D_x^\alpha g\|_{L^{p_2}L^{q_2}} +  \|\,D_t^\beta  f\|_{L^{p_1}L^{q_1}}\| D_x^\alpha g\|_{L^{p_2}L^{q_2}}    \label{mixderivatives} \\
& \hspace{3mm} +  \|\,D_x^\alpha  f\|_{L^{p_1}L^{q_1}}\| D_t^\beta g\|_{L^{p_2}L^{q_2}}    +  \|\,D_t^\beta D_x^\alpha f\|_{L^{p_1}L^{q_1}}\|g\|_{L^{p_2}L^{q_2}}.  \nonumber
\end{align}
for $\alpha, \beta>0$,  $1< p_j, q_j \leq \infty$, $1\leq p,q<\infty$, $\frac {1}{q} = \frac{1}{q_1} + \frac{1}{q_2}$,   and 
$\frac {1}{p} = \frac{1}{p_1} + \frac{1}{p_2}$.
The authors also state that the result hold in higher dimensions. In the case of Lebesgue spaces the analog mixed derivative version of \eqref{classical} was previously studied by Muscalu et al \cite{MPTT}.

Using different methods, Di Plinio and Ou  \cite{DO} prove some multiplier results which implicitly allow them to extended \eqref{mixderivatives} to the case $1/2<p<\infty$ provided
$\alpha>\max(0,\frac{1}{p}-1)$  and $q\geq1$. Finally, we recently became aware of a new version \cite{BM2} of the work of Benea-Muscalu \cite{BM}, and another preprint \cite{BM3}  by the same authors  treating \eqref{mixderivatives} in the full quasi-Banach space case. The combined results of \cite{BM2} and \cite{BM3} allow for  $1/2<p<\infty$ and $1/2<q<\infty$ under the condition $\alpha, \beta >\max(0,\frac{1}{p}-1, \frac{1}{q}-1)$. Moreover, in an even more recent version \cite{BM4} the same authors reduced the condition on $\beta$ to  $\beta >\max(0, \frac{1}{q}-1)$.

We point out that neither \eqref{mixderivatives} implies \eqref{multipliermixlebesgue} nor the other way around. Our proof of \eqref{primera} is carried out in all dimensions $n$ and allows also $1/2<p,q\leq 1$.  In the context of mixed Lebesgue spaces, the version using full derivatives faces a new technical difficulty that forces us to consider versions of Hardy spaces in the mixed-norm setting. This does not seem to be the case in the mixed derivative situation, where one can iterate some vector valued estimates in $x$ and $t$ in some computations.  We believe our arguments  could be modified to give the mixed derivatives version \eqref{mixderivatives} of Benea-Muscalu for the full range of exponents too, but  we will not carry out such computations here.

We are able to treat multipliers
$\T_{m_\nu}$ with limited amount of regularity by applying some of the tools introduced by Tomita \cite{Tomita}, and further developed by Fujita-Tomita \cite{FT}, Grafakos-Si \cite{GS}, Grafakos-Miyachi-Tomita \cite{GMT}, Miyachi-Tomita \cite{MT}, and Li-Sun \cite{LS} for $\nu=0$, and Chaffee-Torres-Wu \cite{CTW} for $\nu> 0$. The techniques for the boundedness results of multipliers (or rather paraproducts) in  \cite{DO}, \cite{BM2}, and \cite{BM3} are then substantially different from ours. Once the boundedness of certain multiplier operators is established, the Leibniz rules follow by what are now familiar arguments, which also work on mixed Lebesgue spaces. As already mentioned, we follow the proof of the Leibniz rules in \cite{GO}, which also share some features with the ones used in \cite{MS}, \cite{BM2}, and \cite{BM3}, and the ones alluded to in \cite{DO}. One common ingredient is the important log estimate for the translated square function. The arguments given in \cite{GO} for such estimates immediately extend to the mixed-norm situation.

After the definitions in Section 2, all of the results involving multiplier operators in Lebesgue spaces are presented in Section~3.  Our main result there is Theorem \ref{multiplier 1}. We then extend in Section~4 the smoothing and Leibniz rule estimates for $\T_{m_\nu}$ to mixed Lebesgue spaces, proving  in Theorem~\ref{multiplier 2} the analogous of Theorem~\ref{multiplier 1} in this context. The Appendix at the end of this article has a technical estimate involving Hardy spaces in the context of mixed norms, which appears to be new.

\section{Function spaces}

Let $\S(\R^n)$ denote the Schwartz class of smooth, rapidly decreasing functions, with the its standard topology, and $\S'(\R^n)$ be the topological dual of
$\S(\R^n)$.  For a function $f\in\S(\R^n)$, we take for definition of  the Fourier transform the expression given by
\begin{align*}
\widehat f(\xi)=\int_{\R^n}f(x)e^{-ix\cdot\xi}dx,
\end{align*}
and, as usual, extend this definition to $\S'(\R^n)$ by duality.  Let $\S_0(\R^n)$ be the subspace of all $f\in\S(\R^n)$ such that
\begin{align}
\int_{\R^n}f(x)x^\alpha dx=0\label{vanishingmoment}
\end{align}
for all $\alpha\in\N_0^n$.

We have already defined in the introduction  $\widehat{I_s f}(\xi) = |\xi|^{-s}\widehat f(\xi)$, for $0<s<n$; and $\widehat{D^s f}(\xi) = |\xi|^{s}\widehat f(\xi)$, for any $s>0$. These definitions certainly make sense for any function in $\S(\R^n)$. We can extend them to all $s\in \R$ in the same way,
$\widehat{D^s f}(\xi) = |\xi|^{s}\widehat f(\xi)$,  but restricting  $f$ to $\S_0(\R^n)$ when $s<-n$. Note that since in such a case
 $\widehat f(\xi)$ vanishes to infinite order at the origin and so  $D^s$ now maps $\S_0(\R^n)$ continuously into $\S_0(\R^n)$.  Hence we can also extend the definition of $D^s$ to the dual of $\S_0(\R^n)$, which can be identified as the class of distributions $\S'(\R^n)$ modulo polynomials.

Fix a function $\psi\in\S_0(\R^n)$ whose Fourier transform is supported in $1/2<|\xi|<2$ and $\widehat\psi(\xi)>c_0$ for $3/5<|\xi|<5/3$ and for $k\in\Z$
define the Littlewood-Paley operator $$\Delta_kf=\psi_{2^{-k}}*f,$$ where $\psi_{2^{-k}}(x)=2^{kn}\psi(2^kx)$.   We will  call such function a Littlewood-Paley function.

For $0<p,q<\infty$ and $s\in\R$, we recall that the homogeneous Triebel-Lizorkin space $\dot F_p^{s,q}$ is the collection of all $f\in \S'(\R^n)$ (modulo polynomials) such that
\begin{align*}
\|f\|_{\dot F_p^{s,q}}=\left\|\(\sum_{k\in\Z}(2^{sk}|\Delta_kf(x)|)^q\)^\frac{1}{q} \right\|_{L^p}< \infty.
\end{align*}
When this is taken modulo polynomials, it is a Banach space norm if $1\leq p,q <\infty$ and a Banach quasi-norm if either $p$ or $q$ are less than 1.  A choice of a different function $\psi$ with the same properties stated above produce equivalent (quasi-)norms. Furthermore, we define $\dot W^{s,p}$ for $s\in\R$ and $0< p<\infty$ to be the set of all $f\in\S'(\R^n)$ such that $D^s f\in L^p$ with (quasi-)norm $\|\,D^sf\|_{L^p}$, and note that for $1<p<\infty$ and $s\in\R$, one has $\dot{W}^{s,p} = \dot F_p^{s,2}$ with comparable norms. In particular, $\dot F_p^{0,2} = L^p$ for that range of $p$. On the other hand for $0<p\leq1$,  $\dot F_p^{0,2}$ coincides with the Hardy space $H^p$.
The inhomogeneous versions are given by
$F_p^{s,q} = \dot F_p^{s,q} \cap L^p$ and
$ W^{s,p} =  \dot{W}^{s,p} \cap L^p$.
Similarly the homogeneous Besov spaces defined by the (quasi-)norms
\begin{align*}
\|f\|_{\dot B_p^{s,q}}=\(\sum_{k\in\Z}\(2^{sk}\left\|\Delta_kf(x)\right\|_{L^p}\)^q\)^\frac{1}{q}
\end{align*}
 and their inhomogeneous counterparts are given by $B_p^{s,q}= \dot B_p^{s,q} \cap L^p$.

For the purposes of the article, we will only consider the mixed Lebesgue spaces $L^p_tL^q_x(\R \times \R^{n})$, or simply $L^p_tL^q_x(\R^{n+1})$, or $L^pL^q( \R^{n+1})$, for $0<p,q<\infty$, which for us will be defined by the (quasi-)norms
$$
\|f\|_{L^p_tL^q_x(\R^{n+1})} = \left(\int_ {\R}\left(\int_{\R^n}|f(t,x)|^q\;dx\right)^{p/q}dt\right)^{1/p}.
$$
We could obtain, of course, versions of our results in mixed Lebesgue spaces defined by a different ordering of the variables, but we just consider the above one because of the significance in applications in partial differential equations.

If  $\psi$  has the same properties as before but in $\R^{n+1}$ and  $1<p,q<\infty$,  it also holds (see \cite{TW}) that
\begin{equation}\label{twpq}
\|f\|_{L^p_tL^q_x(\R^{n+1})} \approx \left\|\(\sum_{k\in\Z}|\Delta_kf|^2\)^\frac{1}{2} \right\|_{L^p_tL^q_x(\R^{n+1})}.
\end{equation}
Finally, we will need a mixed-norm version of the Hardy spaces.   For $0<p,q<\infty,$ the mixed Hardy space $H^{p,q}(\rr^{n+1})$ is defined to be the collection
of all $f\in \S'(\rr^{n+1})$ (modulo polynomials) such that
$$
\|f\|_{H^{p,q}(\rr^{n+1})} =   \left\|\(\sum_{k\in\Z}|\Delta_kf|^2\)^\frac{1}{2} \right\|_{L^p_tL^q_x(\R^{n+1})}<\infty.
$$
Clearly, by definition and \eqref{twpq}, $H^{p,q}(\R^{n+1})=L^pL^q(\R^{n+1})$ whenever $1<p,q<\infty$.  When either $p$ or $q$ is less than or equal to one, there appears to be much less known about other properties of these spaces.  We do mention that a different definition was given by  Cleanthous-Georgiadis-Nielsen \cite{CGN} using non-tangential maximal functions. They showed that their mixed Hardy spaces also coincide with mixed Lebesgue spaces when both indices are larger than one. We do not know if such mixed Hardy spaces coincide with the $H^{p,q}(\R^{n+1})$ above for other values of $p$ and $q$, but it is likely.  Also, a wavelet characterization of $H^{p,q}$ as defined above was obtained in  Georgiadis-Johsen-Nielsen \cite{GJN}. In any case, for our purposes, what we need is the following estimate.
If $0<q,p<\infty$ and $f\in H^{p,q}(\R^{n+1})\cap L^2(\R^{n+1})$, then
\begin{align}\label{aim10}
\|f\|_{L^{p}L^q(\R^{n+1}) }\le C_{p,q}\|f\|_{H^{p,q}(\R^{n+1})}.
\end{align}
Although the case $p=q$ is well known, we could not locate in the literature the case $p\neq q$ when either exponent is smaller or equal to one. We find this case to be rather non-trivial and we provide a proof in the Appendix.

\section{Bilinear multipliers on Lebesgue spaces}

Our first result is concerned with the bilinear Fourier multipliers of the form
$$\T_{m_\nu}(f,g)(x)=\int_{\rr^{2n}}m_\nu(\xi,\eta)e^{i(\xi+\eta)x}\w f(\xi)\w g(\eta)d\xi\,d\eta $$
for $0\le \nu<2n,$ and $f,g\in \S(\R^n)$ where $m_\nu$ satisfies the size condition
\begin{equation}\label{size}
|m_\nu(\xi,\eta)|\lesssim (|\xi|+|\eta|)^{-\nu}.
\end{equation}
Note that the size condition \eqref{size} guarantees that the operators are well-defined and the integral is absolutely convergent.
However, the multipliers $\T_{m_\nu}$ are not a priori bounded on Lebesgue spaces without regularity on $m_\nu$.

We will need the following auxiliary functions. Let $\mathcal{M}_\nu(\R^n)$ be the collection of all sequences  of functions $\{\Phi_\nu^k\}_{k\in \Z}$ satisfying
$\supp\, \Phi_\nu^k\subset \{|(\xi,\eta)|\approx 2^k\}$  and
\begin{equation}\label{3.2}
|\partial^\beta_\xi\partial^{\gamma}_\eta \Phi_\nu^k(\xi,\eta)| \le C_{\beta,\gamma} (|\xi|+|\eta|)^{\nu-|\beta|-|\gamma|},
\end{equation}
for all $(\xi,\eta)\not=0$ and all multi-indices $\beta,\gamma\in \nn_0^n$,
where $C_{\alpha,\beta}$ is a constant independent of $k$.  A typical example is
$\{\Phi_\nu^k\}:=\{(|\xi|^2+|\eta|^2)^{\nu/2} \phi(2^{-k}\xi,2^{-k}\eta)\}$,
where $\phi$ is a Schwartz function supported in $\{|(\xi,\eta)|\approx 1\}$.

The following result provides a sufficient condition for $\T_{m_\nu}$ to be smoothing. In the case $\nu=0$ and $m_0=1$ it is just the Leibniz rule \eqref{classical} with the same range of exponents in \cite{GO} and \cite{MS}.

\begin{theorem}\label{main1}
Let $m_\nu$ be a  multipliers satisfying \eqref{size} for some
$0\le\nu<2n$, and let $1<p_1,p_2< \infty$ and $p$ be such that $1/p_1+1/p_2=1/p$.
Suppose that
\begin{enumerate}
  \item[(i)] for any $\Phi_\nu \in C^\infty(\R^{2n}\setminus \{0\})$ satisfying \eqref{3.2},
  $\T_{m_\nu\Phi_\nu}$ is bounded from $L^{p_1}\times L^{p_2}$ to $L^p$ with norm $C_1$, and
  \item[(ii)] for any $\{\Phi_\nu^k\}\in {\M}_\nu(\R^{n})$, $\{f_k\}\in L^{p_1}(\ell^2)$, and  $\{g_k\}\in L^{p_2}(\ell^2)$
$$
\|\{\T_{m_\nu\Phi_\nu^k}(f_k,g_k)\}_{k\in \Z}\|_{L^p(\ell^1)} \le C_2\|\{f_k\}_{k\in\Z}\|_{L^{p_1}(\ell^2)}\|\{g_k\}_{k\in\Z}\|_{L^{p_2}(\ell^2)}.
$$
\end{enumerate}
Then for $s\in 2\nn$ or $s>\max(0,\frac{n}{p}-n)$, and $f,g\in \S(\rr^{n})$,
\begin{equation}\label{new3.5}
    \|\T_{m_\nu}(f,g)\|_{\dot{W}^{s,p}} \le C' (\|f\|_{\dot{W}^{s-\nu,p_1}}\|g\|_{L^{p_2} }+\|f\|_{L^{p_1}}\|g\|_{\dot{W}^{s-\nu,p_2}}).
 \end{equation}

Moreover, if $C_1,C_2\lesssim A_{m_\nu}$ for some quantity $A_{m_\nu}$ depending on $m_\nu$, then $C'\lesssim A_{m_\nu}$.
\end{theorem}

The proof of Theorem \ref{main1}  needs the following version of the Littlewood-Paley estimate that we take from \cite{GO}.

\begin{lemma}[\cite{GO}]\label{lem:4.2}
Let $\mathbf{m}\in \zz^n\backslash \{0\}$ and $\psi^{\mathbf{m}}(x)=\psi(x+\mathbf{m})$
for some Schwartz function $\psi$ whose Fourier transform is supported in the annulus $1/2\le |\xi|\le 2.$
Let $\Delta_j^{\mathbf{m}}(f)=\Psi_{2^{-j}}^{\mathbf{m}}*f$. Then for $1<p<\infty$ there is a constant $C=C(n,p)$ such that
\begin{equation}\label{LPinL}
\left\|\left(\sum_{j\in \zz}|\Delta_j^{\mathbf{m}}(f)|^2\right)^{1/2}\right\|_{L^p(\rr^n)}\le C \ln(1+| \mathbf m|) \|f\|_{L^p(\rr^n)}.
\end{equation}
\end{lemma}

We note for further use that the proof of this result in \cite{GO} is based, as in the classical case, on the vector valued singular integral
$$
Tf(x)= \int_{\R^n} K(x-y)f(y) dy= \left\{ \int_{\R^n} \Psi_{2^{-j}}(x-y)f(y)dy\right\}_j
$$
as an operator from $L^p(\R^n,\C) \to L^p(\R^n,\ell^2)$ and showing that the kernel satisfies the H\"ormander integral condition
$$
 \int_{|x|>2|y|} \|K(x-y) - K(x)\|_{\C\to \ell^2}\, dx \leq
\sum_{j\in\Z} \int_{|x|>2|y|} |\Psi_{2^{-j}}(x-y) - \Psi_{2^{-j}}(x)|\, dx
$$
$$
\leq C \ln(1+ |\mathbf m|).
$$

The same of course holds if we work in $\R^{n+1}$. But then, by the  results in Benedek-Calder\'on-Panzone \cite{BCP} (see also \cite{TW}), the boundedness
$$
T:L^p(\R^{n+1},\C) \to L^p(\R^{n+1},\ell^2)
$$
for all $1<p<\infty$, also gives
$$
T:L^p_tL^q_x(\R^{n+1},\C) \to L^p_tL^q_x(\R^{n+1},\ell^2)
$$
for all $1<p,q<\infty$, and hence the bound
\begin{equation}\label{LPmixed}
\left\|\left(\sum_{j\in \zz}|\Delta_j^{\mathbf{m}}(f)|^2\right)^{1/2}\right\|_ {L^p_tL^q_x(\R^{n+1})}\lesssim \ln(1+ |\mathbf m|) \|f\|_{L^p_tL^q_x(\R^{n+1})}
\end{equation}
for $\mathbf m \in \Z^{n+1}$.

\noindent\textit{Proof of Theorem \ref{main1}.}
We follow very closely the arguments in \cite[Theorem 1]{GO}.  Select a function $\w{\phi}\in \S(\rr^n)$ such that $\supp\,\widetilde{\psi} \subset B(0,2)$,
$\w{\phi}(\xi)=1$ on $|\xi|\le 1$, and let
$\w\psi(\xi)=\w\phi(\xi)-\w\phi(2\xi)$   so that
$$\sum_{j\in\zz}\w \psi(2^{-j}\xi)=1\quad \mbox{for}\ \xi\not=0.$$
Let also $\w{\widetilde{\psi}}(\xi)=\sum_{j:|j|\le2}\w{\psi}(2^{j}\xi)$.

We use a familiar paraproduct decomposition to write $D^sT_\nu(f,g)$ as
\begin{align*}
& D^s\T_{m_\nu}(f,g)(x) \\
& =\sum_{j,k\in \zz} \int_{\rr^{2n}}e^{i(\xi+\eta)x} m_\nu(\xi,\eta) |\xi+\eta|^s \w\psi(2^{-j}\xi)\w
f(\xi)\w\psi(2^{-k}\xi)\w g(\eta)d\xi\,d\eta\\
&=\sum_{j\in \zz} \int_{\rr^{2n}}e^{i(\xi+\eta)x} m_\nu(\xi,\eta)\frac{|\xi+\eta|^s}{|\xi|^{s-\nu}}\w\psi(2^{-j}\xi)\w {D^{s-\nu}f}(\xi)\w{\phi}(2^{-j+3}\xi)\w g(\eta)d\xi\,d\eta\\
&\quad+\sum_{j\in \zz } \int_{\rr^{2n}}e^{i(\xi+\eta)x} m_\nu(\xi,\eta)\frac{|\xi+\eta|^s}{|\eta|^{s-\nu}}\w{\phi}(2^{-j+3}\xi)\w f(\xi)\w\psi(2^{-j}\xi)\w {D^{s-\nu}g}(\eta)d\xi\,d\eta\\
&\quad+\sum_{\substack{j,k\in \Z\\ |j-k|\le 2}} \int_{\rr^{2n}}e^{i(\xi+\eta)x} m_\nu(\xi,\eta)\frac{|\xi+\eta|^s}{|\eta|^{s-\nu}}\w\psi(2^{-j}\xi)\w f(\xi)\w\psi(2^{-k}\xi)\w {D^{s-\nu}g}(\eta)d\xi\,d\eta\\
&= T_1(D^{s-\nu}f,g)(x)+T_2(f,D^{s-\nu}g)(x)+T_3(f,D^{s-\nu}g)(x),
\end{align*}
where $T_i$ for $i=1,2,3$ are defined via the bilinear symbols
\begin{align*}
m_1(\xi,\eta)&=m_\nu(\xi,\eta)\frac{|\xi+\eta|^s}{|\xi|^{s-\nu}}\sum_{j\in\zz}\w\psi(2^{-j}\xi) \w \phi(2^{-j+3}\eta) := m_\nu(\xi,\eta)\Phi_\nu^{(1)}(\xi,\eta) ,\\
m_2(\xi,\eta)&=m_\nu(\xi,\eta)\frac{|\xi+\eta|^s}{|\eta|^{s-\nu}}\sum_{j\in\zz}\w\phi(2^{-j+3}\xi) \w \psi(2^{-j}\eta):= m_\nu(\xi,\eta)\Phi_\nu^{(2)}(\xi,\eta) ,\\
m_3(\xi,\eta)&=m_\nu(\xi,\eta)\frac{|\xi+\eta|^s}{|\eta|^{s-\nu}}\sum_{j\in\zz}\w\psi(2^{-j}\xi) \w {\widetilde{\psi}}(2^{-j}\eta):= m_\nu(\xi,\eta)\Phi_\nu^{(3)}(\xi,\eta) .
\end{align*}
It will be enough then to show that for $i=1,2,3$,
\begin{equation}\label{Ti}
\|T_i(f,g)\|_{L^{p}} \lesssim  \|f\|_{L^{p_1}}\|g\|_{L^{p_2} }.
\end{equation}

For $i=1,2$, the functions  $\Phi_\nu^{(i)}$ satisfy \eqref{3.2} since $|\xi+\eta|\neq 0$ on their support.
 Hence $T_1$ and $T_2$ satisfy \eqref{Ti} by hypothesis. The same is true for $T_3$ if $s$ is an even integer as in such a case $\Phi_\nu^{(3)}$ still satisfies \eqref{3.2} .

Let $\w\varphi \in C_c^\infty(\R^n)$ be a function which has slightly larger compact support than $\w\psi$
and satisfies $\w\psi=\w\psi\w\varphi$. By looking carefully at the support of the functions involved in the following integral and the properties of $\phi$, we may write
\begin{align*}
\begin{split}
& T_3(f,g)(x)\\
&=\sum_{k\in \zz}\int_{\rr^{2n}}  \w\phi (2^{-k-4}(\xi+\eta)) \frac{|\xi+\eta|^s}{|\eta|^{s-\nu}} m_\nu(\xi,\eta)\w{\varphi}(2^{-k}\xi)\w{\widetilde\varphi}(2^{-k}\eta) \w{\psi}(2^{-k}\xi)\w f(\xi)\\
  &\quad \quad \quad \times\w {\widetilde\psi}(2^{-k}\eta) \w g(\eta)
  e^{ i(\xi+\eta)x}d\xi\,d\eta \\
&=\sum_{k\in \zz} 2^{2nk}\int_{\rr^{2n}}  \w\phi_s (2^{-4}(\xi+\eta)) m_\nu(2^k\xi,2^k\eta)\Phi_\nu^{(3),k}(2^k\xi,2^k\eta) \w{\psi}(\xi)\w f(2^k\xi)\\
  &\quad \quad \quad \times \w {\widetilde{\psi}}_{-s}(\eta) \w g(2^k\eta)
  e^{ i2^k(\xi+\eta)x}d\xi\,d\eta,
\end{split}
\end{align*}
where
$\w {\widetilde{\psi}}_{-s}(\eta)= |\eta|^{-s}\w\psi(\eta)$, $\w\phi_s(\xi)=|\xi|^s\w\phi(\xi)$, and
$$\{\Phi_\nu^{(3),k}(\xi,\eta)\}=\{|\eta|^\nu\w{\varphi}(2^{-k}\xi)\w{\widetilde\varphi}(2^{-k}\eta)\}\in \M_\nu.$$

Since $\w\phi_s (2^{-4}(\cdot))$ has compact support, there exists a constant $c_0$ so that it can be expanded in a Fourier series on  a cube centered at the origin of side length $c_0$ and obtain
\begin{equation}\label{eq:43}
\w\phi_s (2^{-4}(\xi+\eta))\w\psi(\xi)\w {\widetilde{\psi}}_{-s}(\eta)=\sum_{\mathbf{m}\in \zz^n} C_{\mathbf{m}}^s e^{\frac{2\pi i}{c_0}(\xi+\eta)\mathbf{m}} \w\psi(\xi)\w {\widetilde{\psi}}_{-s}(\eta),
\end{equation}
where the Fourier coefficients satisfy
\begin{equation*}
  C_{\mathbf{m}}^s=O((1+|\mathbf{m}|)^{-s-n})\quad\mbox{as}\quad |\mathbf{m}|\rightarrow \infty.
\end{equation*}

Therefore
\begin{align*}\label{eq3.9}
& |T_3(f,g)(x)| \\
&\le\sum_{k\in\zz}  2^{2nk} \left| \int_{\rr^{2n}} \sum_{\mathbf{m}\in \zz^n}C_{\mathbf{m}}^s m_\nu(2^k\xi,2^k\eta)\Phi_\nu^{(3),k}(2^k\xi,2^k\eta) e^{\frac{2\pi i}{c_0}(\xi+\eta)\mathbf{m}} \right. \\
&\left. \phantom{\int}\qquad \times\w{\psi}(\xi)\w f(2^k\xi)\w {\widetilde{\psi}}_{-s}(\eta) \w g(2^k\eta)  e^{ i2^k(\xi+\eta)x}d\xi\,d\eta\right|\\
  &\leq \sum_{\mathbf{m}\in \zz^n}|C_{\mathbf{m}}^s| \sum_{k\in\zz}  2^{2nk}\left|\int_{\rr^{2n}} m(2^k\xi,2^k\eta)\Phi_\nu^{(3),k}(2^k\xi,2^k\eta)  e^{\frac{2\pi i}{c_0}(\xi+\eta)\mathbf{m}}\right. \\
&\left.\phantom{\int}\qquad \times\w{\psi}(\xi)\w f(2^k\xi)\w {\widetilde{\psi}}_{-s}(\eta) \w g(2^k\eta)
  e^{ i2^k(\xi+\eta)x}d\xi\,d\eta\right| \\
&=\sum_{\mathbf{m}\in \zz^n}|C_{\mathbf{m}}^s| \sum_{k\in\zz} \left| \int_{\rr^{2n}}  m(\xi,\eta)\Phi_\nu^{(3),k}(\xi,\eta)\widehat{\Delta_k^{\mathbf m}(f)}(\xi)\widehat{\widetilde{\Delta}_k^{\mathbf m}(g)}(\eta)\right.\\
 &\left.\phantom{\int}\qquad \times\ e^{2\pi i(\xi+\eta)x}d\xi\,d\eta\right|\\
&=  \sum_{\mathbf{m}\in \zz^n}|C_{\mathbf{m}}^s| \sum_{k\in\zz} | \T_{m\Phi_\nu^{(3),k}}(\Delta_k^{\mathbf m}(f), \widetilde{\Delta}_k^{\mathbf m}(g))(x) |,
\end{align*}
where $$\widehat{\Delta_k^{\mathbf m}(f)}(\xi)=e^{\frac{2\pi i}{c_0}2^{-k}\xi \mathbf{m}} \w{\psi}(2^{-k}\xi)\w f(\xi),$$
$$
\widehat{\widetilde{\Delta}_k^{\mathbf m}(g)}(\eta)=e^{\frac{2\pi i}{c_0}2^{-k}\eta \mathbf m}\w {\widetilde{\psi}_{-s}}(2^{-k}\eta) \w {g}(\eta).$$

Let $p_*=\min(p,1)$. By the $L^{p_1}(\ell^2)\times L^{p_2}(\ell^2)\rightarrow L^p(\ell^1)$ boundedness of $\T_{m\Phi_\nu^{(3),k}}$  and Lemma \ref{lem:4.2},
 we obtain
\begin{align*}
&  \left\|\sum_{\mathbf{m}\in \zz^n}|C_{\mathbf{m}}^s| \sum_{k\in\zz} | \T_{m\Phi_\nu^{(3),k}}(\Delta_k^{\mathbf m}(f), \widetilde{\Delta}_k^{\mathbf m}(g))(x) |\right\|_{L^p}^{p_*}\\
&  \lesssim  \sum_{\mathbf{m}\in \zz^n}|C_{\mathbf{m}}^s|^{p_*}  \left\| \(\sum_{k\in\Z}|\Delta_k^{\mathbf m}(f)|^2\)^{1/2}\right\|_{L^{p_1}}^{p_*}
  \left\|\(\sum_{k\in\Z}|\widetilde{\Delta}_k^{\mathbf m}(g)|^2\)^{1/2}\right\|_{L^{p_2}}^{p_*}\\
&  \lesssim  \sum_{\mathbf{m}\in \zz^n}|C_{\mathbf{m}}^s|^{p_*} [\ln(1+|\mathbf{m}|)]^{2p_*} \left\| f\right\|_{L^{p_1}}^{p_*}
  \left\| g \right\|_{L^{p_2}}^{p_*}\\
&  \lesssim  \left\| f\right\|_{L^{p_1}}^{p_*}
  \left\| g \right\|_{L^{p_2}}^{p_*},
\end{align*}
since the series $\sum_{\mathbf{m}\in \zz^n}|C_{\mathbf{m}}^s|^{p_*} [\ln(1+ \mathbf m)]^{2p_*}$ converges under our assumption $p_*(n+s)>n$.
This completes the proof of Theorem \ref{main1}.
$\hfill\Box$

As applications, we shall prove the smoothing property of bilinear fractional Fourier multipliers with limited regularity (including the Coifman-Meyer bilinear fractional multipliers with that type of regularity).
Let $\Psi\in \mathcal{S}(\mathbb{R}^{2n})$ be such that
\begin{equation}\label{approx1}
  \mbox{supp} \ \Psi\subset \{\xi\in \mathbb{R}^{2n}: 1/2\le |(\xi,\eta)|\le 2\},
\quad \sum_{k\in \mathbb{Z}} \Psi(\xi/2^k,\eta/2^k)=1
\end{equation}
for all $(\xi,\eta) \in \mathbb{R}^{2n} \setminus \{0\}$.

For a function $m$,
 $\nu\ge 0$, and $k\in \mathbb{Z}$, define
\begin{equation*}
  m_k^\nu(\xi,\eta)=2^{k\nu}m(2^k\xi,2^k\eta)\Psi(\xi,\eta).
\end{equation*}
We will consider bilinear Fourier multipliers $T_\nu$ with symbols $m$ satisfying the Sobolev regularity studied in \cite{GMT}:
\begin{equation}\label{smooth}
\sup_{k\in \zz}\|m_k^\nu\|_{W^{(r,r),2}(\rr^{2n})}<\infty.
\end{equation}
Here $W^{(r_1,r_2),2}(\rr^{2n})$ denotes the product Sobolev space consisting of all functions $h$  in $L^2(\rr^{2n})$ satisfying
 $$\|h\|_{W^{(r_1,r_2),2}}:=\left(\int_{\rr^{n}}\int_{\rr^{n}}(1+|x|^2)^{r_1}(1+|y|^2)^{r_2}|\widehat{h}(x,y)|^2dxdy\right)^{1/2}<\infty.$$

We will apply Theorem \ref{main1} to show that the bilinear Fourier multipliers $T_m$ satisfying \eqref{smooth} satisfy Leibniz-type rules and in particular are smoothing when $\nu>0$.

\begin{theorem}\label{multiplier 1}
Suppose that $m$ satisfies \eqref{smooth} for some $0\le\nu<2n$ and $n/2<r\le n$. Let
$n/r<p_i<\infty$, $i=1,2,$ and $1/p=1/p_1+1/p_2$.
Then for $s\in 2\nn$ or $s>\max(0,\frac{n}{p}-n)$, the bilinear multiplier with symbol $m$ satisfies
\begin{align*}  \|\T_{m}(f,g)\|_{\dot{W}^{s,p}}
\lesssim \sup_{j\in \zz}\|m_j^\nu\|_{W^{(r,r),2}} \left(\|f\|_{\dot{W}^{s-\nu,p_1}}\|g\|_{L^{p_2} }+\|f\|_{L^{p_1}}\|g\|_{\dot{W}^{s-\nu,p_2}}\right)
\end{align*}
for all $f,g$ in $\mathcal{S}(\rr^n)$.
\end{theorem}

We remark again that for $\nu=0$ this is essentially the Leibniz result  in \cite{MS} and \cite{GO} (and the work of other authors for a smaller range of exponents) which corresponds to  the multiplier $m=1$, except that we allow also for multipliers with minimal smoothness.

\begin{corollary}[\cite{MS, GO}]
Let $1<p_1,p_2<\infty$ and $1/p=1/p_1+1/p_2$. Then
$$
  \|f\cdot g\|_{\dot{W}^{s,p}} \lesssim \|f\|_{\dot{W}^{s,p_1}}\|g\|_{L^{p_2} }+\|f\|_{L^{p_1}}\|g\|_{\dot{W}^{s,p_2}}
$$
for  $s\in 2\nn$ or $s>\max(0,\frac{n}{p}-n)$ and all $f,g\in \mathcal{S}(\rr^n)$.
\end{corollary}

In addition, it follows from the results in \cite{N} that the multipliers of order zero we are considering satisfy also
$$
\|\T_{m_0} (f,g) \|_{B^{s, p}_q}  \lesssim      \|f\|_{B^{s, p_1}_q}  \|g\|_{L^{p_2}} +   \|f\|_{L^{p_1}} \|g\|_{B^{s, p_2}_q}
$$
for $1 <p_1, p_2 <\infty$, $1 <p <\infty$, $1/p=1/p_1 + 1/p_2$,  $s>\max (0,\frac{n}{p}-n)$, and $0<q\leq \infty$, since they are bounded from $L^{p_1} \times L^{p_2}$ into $L^p$.

The case $\nu>0$ of Theorem \ref{multiplier 1} gives the smoothing of the bilinear fractional integral operators.

\begin{corollary}
Let $0<\nu<2n$ and let $1<p_1,p_2<\infty$ and $1/p=1/p_1+1/p_2$. Then for $f,g\in \mathcal{S}(\rr^n)$,
$$
  \|\I_\nu(f,g)\|_{\dot{W}^{s,p}} \lesssim \|f\|_{\dot{W}^{s-\nu,p_1}}\|g\|_{L^{p_2} }+\|f\|_{L^{p_1}}\|g\|_{\dot{W}^{s-\nu,p_2}}
$$
for  $s\in 2\nn$ or $s>\max(0,\frac{n}{p}-n)$.
\end{corollary}
The proof of Theorem \ref{multiplier 1} needs the following lemma.
\begin{lemma}\cite{FT,GMT}\label{lem: maximal}
Let $R>0$, $r>n/2$, and
$\max\{1,\frac{n}{r}\}<l<2$. Then there exists a constant $C>0$
such that
\begin{align*}
\left|\int_{\mathbb{R}^{2n}} 2^{2jn}  \widehat{\sigma}(2^{j}(x-y_1), 2^j(x-y_2))
f(y_1) g(y_2)dy_1 \, dy_2\right|\\
\le C\|\sigma\|_{W^{(r,r),2}} (M(|f|^l))^{1/l}(x)(M(|g|^l))^{1/l}(x)
\end{align*}
for all $j\in \mathbb{Z}, \sigma\in W^{(r,r),2}(\mathbb{R}^{2n})$
with supp $\sigma\subset \{\sqrt{|\xi|^2+ |\eta|^2}\le R\}$
and $f, g\in \mathcal{S}(\mathbb{R}^n).$
\end{lemma}

\noindent\textit{Proof of Theorem \ref{multiplier 1}.}
We need to show that $m$ satisfies the assumptions of Theorem \ref{main1}.

We first note that the size condition $|m(\xi,\eta)|\lesssim (|\xi|+|\eta|)^{-\nu}$ follows from our hypothesis $m_k^\nu \in W^{(r,r),2}$.
In fact, let $$m_0(\xi,\eta)=m(\xi,\eta)|(\xi,\eta)|^\nu.$$
It suffices to verify that $|m_0(2^j\xi, 2^j\eta)|$ is bounded uniformly in $j$ for $(\xi,\eta)$ satisfying $1/2\le |(\xi,\eta)|\le 2$. Let $\Psi$ be a function satisfying \eqref{approx1}. Then for  $1/2\le |(\xi,\eta)|\le 2$,
\begin{align*}
|m_0(2^j\xi,2^j\eta)|\le \sum_{-1\le l\le 1}|m_0(2^j\xi,2^j\eta)\Psi(2^l\xi,2^l\eta)|.
\end{align*}
We estimate the term for $l=0$ as the other ones can be treated in exactly the same way.
 For $1/2\le |(\xi,\eta)|\le 2$,
 \begin{align*}
|m_0(2^j\xi,2^j\eta)  \Psi(\xi,\eta)| &
\approx |2^{j\nu}m(2^j\xi,2^j\eta)\Psi(\xi,\eta)| \\
&\approx  \int_{\rr^{2n}}| (2^{j\nu}m(2^j\cdot,2^j\cdot)\Psi)^{\wedge}(x,y)| dxdy\\
&\le \int_{\rr^{2n}} (1+|x|^2)^{-r/2}(1+|y|^2)^{-r/2}  \\
&  \,\,\,\,\,\times (1+|x|^2)^{r/2}(1+|y|^2)^{r/2} |(2^{j\nu}m(2^j\cdot)\Psi)^{\wedge}(x,y)|dxdy\\
& \lesssim \sup_{k\in \zz}\|m_k^\nu\|_{W^{(r,r),2}}<\infty.
\end{align*}

 Next, for a real number $r$, denote by $\lfloor r \rfloor$ the greatest integer function of $r$.
  Since $W^{(r,r),2}$ is a multiplication algebra, for any $\Phi_\nu$ satisfying \eqref{3.2} and $\widetilde{\Psi}$ in $C_0^\infty(\R^{2n})$ with $\Psi=\Psi\widetilde{\Psi}$, we have
\begin{align*}
 \|(m\Phi_\nu)_k^0\|_{W^{(r,r),2}}&=\|m(2^k\xi,2^k\eta)\Phi_\nu(2^k\xi,2^k\eta)\Psi(\xi,\eta)\|_{W^{(r,r),2}}\\
&=\|m(2^k\xi,2^k\eta)\Phi_\nu(2^k\xi,2^k\eta)\Psi(\xi,\eta)\widetilde{\Psi}(\xi,\eta)\|_{W^{(r,r),2}}\\
&\le \|2^{k\nu}m(2^k\xi,2^k\eta)\Psi(\xi,\eta)\|_{W^{(r,r),2}} \\
&\hspace{3cm}\times \|2^{-k\nu}\Phi_\nu(2^k\xi,2^k\eta)\widetilde{\Psi}(\xi,\eta)\|_{W^{(r,r),2}}\\
&\lesssim \|m_k^\nu\|_{W^{(r,r),2}}  \|2^{-k\nu}\Phi_\nu(2^k\xi,2^k\eta)\widetilde{\Psi}(\xi,\eta)\|_{W^{(\lfloor r \rfloor+1,\lfloor  r \rfloor+1),2}}\\
&\lesssim \|m_k^\nu\|_{W^{(r,r),2}},
\end{align*}
since it is easy to verify using \eqref{3.2} that
$$\|2^{-k\nu}\Phi_\nu(2^k\xi,2^k\eta)\widetilde{\Psi}(\xi,\eta)\|_{W^{(\lfloor  r \rfloor+1,\lfloor  r \rfloor+1),2}}\lesssim1.$$
It follows that  $\T_{m\Phi_\nu}$ is a bilinear Fourier multiplier studied by Miyachi-Tomita in \cite{MT},
and hence the $L^{p_1}\times L^{p_2}\rightarrow L^p$ boundedness (with norm controlled by $\sup_k\|m_k^\nu\|_{W^{(r,r),2}}$).
Finally, given $\{\Phi_\nu^k\}\in {\M}_\nu(\R^n)$
set $$m_k(\xi,\eta)=m(2^k\xi,2^k\eta)\Phi_\nu^k(2^k\xi,2^k\eta)$$ for $k\in \zz$.
Then,

\begin{align*}
& \T_{m\Phi_\nu^k}(f_k,g_k) \\
& = \int_{\R^{2n}} e^{i(\xi+\eta)x}m(\xi,\eta) \Phi_\nu^k(\xi,\eta)\w {f_k}(\xi) \w {g_k}(\eta)d\xi d\eta\\
&=\int_{\R^{2n}} e^{i(\xi+\eta)x}m_k(2^{-k}\xi,2^{-k}\eta) \w {f_k}(\xi) \w {g_k}(\eta)d\xi d\eta\\
&\approx2^{2kn}\int_{\R^{2n}} \mathcal{F}^{-1}m_k(2^{k}(x-y_1),2^{k}(x-y_2))  {f_k}(y_1)  {g_k}(y_2)dy_1 dy_2,
\end{align*}
where
$$
m_k(\xi,\eta)=m(2^{k}\xi,2^{k}\eta) \Phi_\nu^k(2^{k}\xi,2^{k}\eta).
$$
Choose $k_0\in \zz_+$ such that supp~$\Phi_\nu^k\subset \{2^{k-k_0}\le |(\xi,\eta)|<2^{k+k_0}\}$.
Using again that $W^{(r,r),2}$ is a multiplication algebra and a function $\tilde \Psi$ as before,
\begin{align*}
& \|m_k\|_{W^{(r,r),2}}  \\
& \le \sum_{j=-k_0}^{k_0+1}\|m(2^k\cdot)\Phi_\nu^k(2^k\cdot)\Psi(2^{-j}\cdot) \|_{W^{(r,r),2}}\\
&\le \sum_{j=-k_0}^{k_0+1} 2^{\max (0,j)2t} 2^{-jn} \|m(2^{k+j}\cdot)\Phi_\nu^k(2^{k+j}\cdot)\Psi \|_{W^{(r,r),2}}\\
&\lesssim \sum_{j=-k_0}^{k_0+1}\|2^{(k+j)\nu} m(2^{k+j}\cdot) \Psi\|_{W^{(r,r),2}} \|2^{-(k+j)\nu}\Phi_\nu^k(2^{k+j}\cdot)  \tilde \Psi \|_{W^{(r,r),2}} \\
&\lesssim \sup_{k\in\zz}\|m_k^\nu \|_{W^{(r,r),2}} \sum_{j=-k_0}^{k_0+1}\|2^{-(k+j)\nu}\Phi_\nu^k(2^{k+j}\cdot)
\tilde \Psi  \|_{W^{(\lfloor r \rfloor+1,\lfloor r \rfloor+1),2}} \\
&\lesssim \sup_{k\in\zz}\|m_k^\nu \|_{W^{(r,r),2}}.
\end{align*}
Applying Lemma \ref{lem: maximal}, we obtain
\begin{equation}\label{3.8}
|\T_{m\Phi_\nu^k}(f,g)(x)|
 \lesssim \sup_{j\in\zz}\|m_j^\nu \|_{W^{(r,r),2}} (M(|f|^l))^{1/l}(x)(M(|g|^l))^{1/l}(x)
\end{equation}
for $l\in (n/r,\min\{2,p_1,p_2\})$.
From this, Cauchy-Schwarz inequality and the Fefferman-Stein \cite{FS} vector-valued maximal estimate, we have
\begin{align*}
 & \left\|\sum_k|\T_{m\Phi_\nu^k}(f_k,g_k)|\right\|_{L^{p}}    \\
&  \lesssim  
A_m \left\|\(\sum_k ((M(|f_k|^l))^{2/l}\)^{1/2} \(\sum_k ((M(|g_k|^l))^{2/l}\)^{1/2}\right\|_{L^{p}}\\
    &  \lesssim  
A_m    \left\|\(\sum_k ((M(|f_k|^l))^{2/l}\)^{1/2}\right\|_{L^{p_1}}
    \left\|\(\sum_k ((M(|g_k|^l))^{2/l}\)^{1/2}\right\|_{L^{p_2}}\\
    &  \lesssim  
  A_m  \left\|\(\sum_k ((M(|f_k|^l))^{2/l}\)^{l/2}\right\|^{1/l}_{L^{p_1/l}}
    \left\|\(\sum_k ((M(|g_k|^l))^{2/l}\)^{l/2}\right\|^{1/l}_{L^{p_2/l}}\\
    &\lesssim 
A_m    \|\{f_k\}\|_{L^{p_1}}\|\{g_k\}\|_{L^{p_2}},
\end{align*}
where $A_m=\sup_{j\in\zz}\|m_j^\nu \|_{W^{(r,r),2}}$.
The result follows now by applying Theorem \ref{main1}. \hfill $\Box$

\section{Bilinear multiplier on mixed Lebesgue spaces}

In this section, we show how to extend the result of the previous one to the context of mixed Lebesgue spaces. We will use the following version of the Fefferman-Stein inequality, which can be found in \cite{Fer} and \cite{K}.

\begin{lemma}[\cite{Fer, K}]\label{maximalpq}
Let $\{f_j\}$ be a sequence of locally integrable functions in $\rr^{n+1}$ and $M$ the Hardy-Littlewood maximal operator also in $\rr^{n+1}$. Then for $1<p,q,r<\infty,$
$$
\left\|\left(\sum_{j}|M(f_j)|^r\right)^{1/r}\right\|_{L^pL^q} \lesssim \left\|\left(\sum_{j}|f_j|^r\right)^{1/r}\right\|_{L^pL^q}.
$$
\end{lemma}

The following is the version of Theorem \ref{main1} in $L^pL^q(\R^{n+1})$.

\begin{theorem}\label{main2}
Let $m_\nu$ be a multiplier satisfying \eqref{size} in $\rr^{n+1}$ for some
$0\le\nu<2(n+1)$, and let $1<p_1,p_2, q_1, q_2< \infty$ and $p$ and $q$ be such that $1/p_1+1/p_2=1/p$ and
$1/q_1+1/q_2=1/q$. Suppose that
\begin{enumerate}
  \item[(i)] for any $\Phi_\nu \in C^\infty(\R^{2(n+1)}\setminus \{0\})$ satisfying \eqref{3.2},
  $\T_{m_\nu\Phi_\nu}$ is bounded from $L^{p_1}L^{q_1}\times L^{p_2}L^{q_2}$ to $L^pL^{q}$ with norm $C_1$, and
  \item[(ii)] for any sequences $\{\Phi_\nu^k\}\in \M_\nu(\R^{n+1})$, $\{f_k\}\in L^{p_1}L^{q_1}(\ell^2)$, and  $\{g_k\}\in L^{p_2}L^{q_2}(\ell^2)$
$$
\|\{\T_{m_\nu\Phi_\nu^k}(f_k,g_k)\}_{k\in \Z}\|_{L^pL^{q}(\ell^1)} \le C_2\|\{f_k\}_{k\in\Z}\|_{L^{p_1}L^{q_1}(\ell^2)}
\|\{g_k\}_{k\in\Z}\|_{L^{p_2}L^{q_2}(\ell^2)}.
$$
\end{enumerate}
Then for $s\in 2\nn$ or $s>\max(0,\frac{n+1}{p}-(n+1), \frac{n+1}{q}-(n+1))$, and $f,g\in \S(\rr^{n+1})$,
\begin{align*}\label{new3.5bis}
   & \|D^s\T_{m_\nu}(f,g)\|_{L^pL^q(\rr^{n+1})}  \le C' (\| D^{s-\nu}f\|_{L^{p_1}L^{q_1}(\rr^{n+1})} \|g\|_{L^{p_2}L^{q_2} (\rr^{n+1})}+\\
 & + \|f\|_{L^{p_1}L^{q_1}(\rr^{n+1})}\|D^{s-\nu}g\|_{L^{p_2}L^{q_2}(\rr^{n+1})  }).
 \end{align*}

 Moreover, if $C_1,C_2\lesssim A_{m_\nu}$ for some quantity $A_{m_\nu}$ depending on $m_\nu$, then $C'\lesssim A_{m_\nu}$.
\end{theorem}

\begin{proof}
Proceeding exactly as in Theorem \ref{main1}, we can write
$$D^s\T_{m_\nu}(f,g)(t,x)= $$$$ T_1(D^{s-\nu}f,g)(t,x)+T_2(f,D^{s-\nu}g)(t,x)+T_3(f,D^{s-\nu}g)(t,x).
$$
The first two terms present no differences from before and are easily bounded.  The same is true for the third one if $s$ is an even integer. Otherwise, we note that all the pointwise estimates used in the proof of Theorem \ref{main1} work in any dimension, so we can arrive at
$$
|T_3(f,g)(t,x)| \le  \sum_{\mathbf{m}\in \zz^{n+1}}|C_{\mathbf{m}}^s| \sum_{k\in\zz} | \T_{m\Phi_\nu^{(3),k}}(\Delta_k^{\mathbf m}(f), \widetilde{\Delta}_k^{\mathbf m}(g))(t,x) |,
$$
where now $\{\Phi_\nu^{(3),k}\}\in \M_\nu(\R^{n+1})$;
$$
\widehat{\Delta_k^{\mathbf m}(f)}(\xi)=e^{\frac{2\pi i}{c_0}2^{-k}\xi \mathbf{m}} \w{\psi}(2^{-k}\xi)\w f(\xi)
$$
and
$$
\widehat{\widetilde{\Delta}_k^{\mathbf m}(g)}(\eta)=e^{\frac{2\pi i}{c_0}2^{-k}\eta \mathbf m}\w {\widetilde{\psi}_{-s}}(2^{-k}\eta) \w {g}(\eta),
$$
whith $\psi$ and $\tilde \psi$  Littlewood-Paley functions in $\R^{n+1}$; and the coefficients $C_{\mathbf{m}}^s$ satisfy
$$
  C_{\mathbf{m}}^s=O((1+|\mathbf{m}|)^{-s-n-1})\quad\mbox{as}\quad |\mathbf{m}|\rightarrow \infty.
$$

We consider two cases.  If $q>1$,  let  again $p_*=\min(p,1)$.  By assumption (ii), $\T_{m\Phi_\nu^{(3),k}}$ is bounded from $L^{p_1}(\ell^2)\times L^{p_2}(\ell^2)$ to $L^p(\ell^1)$, which together with now \eqref{LPmixed} in place of \eqref{LPinL} , yield
\begin{align*}
&  \left\|\sum_{\mathbf{m}\in \zz^{n+1}}|C_{\mathbf{m}}^s| \Big|\sum_{k\in\zz} | \T_{m\Phi_\nu^{(3),k}}(\Delta_k^{\mathbf m}(f), \widetilde{\Delta}_k^{\mathbf m}(g)) \Big|\right\|_{L^pL^q}^{p_*}\\
&  \le\left\|\sum_{\mathbf{m}\in \zz^{n+1}}|C_{\mathbf{m}}^s| \Big\|\sum_{k\in\zz} | \T_{m\Phi_\nu^{(3),k}}(\Delta_k^{\mathbf m}(f), \widetilde{\Delta}_k^{\mathbf m}(g)) \Big\|_{L^q}\right\|_{L^p}^{p_*}\\
&  \le \sum_{\mathbf{m}\in \zz^{n+1}}|C_{\mathbf{m}}^s|^{p_*} \left\| \Big\|\sum_{k\in\zz} | \T_{m\Phi_\nu^{(3),k}}(\Delta_k^{\mathbf m}(f), \widetilde{\Delta}_k^{\mathbf m}(g)) \Big\|_{L^q}\right\|_{L^p}^{p_*}\\
&  \lesssim  \sum_{\mathbf{m}\in \zz^{n+1}}|C_{\mathbf{m}}^s|^{p_*}  \left\| \(\sum_{k\in\Z}|\Delta_k^{\mathbf m}(f)|^2\)^{1/2}\right\|_{L^{p_1}L^{q_1}}^{p_*}
  \left\|\(\sum_{k\in\Z}|\widetilde{\Delta}_k^{\mathbf m}(g)|^2\)^{1/2}\right\|_{L^{p_2}L^{q_2}}^{p_*}\\
&  \lesssim  \sum_{\mathbf{m}\in \zz^{n+1}}|C_{\mathbf{m}}^s|^{p_*} [\ln(1+|\mathbf{m}|)]^{2p_*} \left\| f\right\|_{L^{p_1}L^{q_1}}^{p_*}
  \left\| g \right\|_{L^{p_2}L^{q_2}}^{p_*}\\
&  \lesssim  \left\| f\right\|_{L^{p_1}L^{q_1}}^{p_*}
  \left\| g \right\|_{L^{p_2}L^{q_2}}^{p_*},
\end{align*}
since by the hypothesis on $s$,   $p_*(n+1+s)>n+1$.

If $0<q\le 1,$ we have with $(p/q)_*=\min(p/q,1)$
\begin{align*}
&  \left\|\sum_{\mathbf{m}\in \zz^{n+1}}|C_{\mathbf{m}}^s| \Big|\sum_{k\in\zz} | \T_{m\Phi_\nu^{(3),k}}(\Delta_k^{\mathbf m}(f), \widetilde{\Delta}_k^{\mathbf m}(g)) \Big|\right\|_{L^pL^q}^{q(\frac{p}{q})_*}\\
&  \le\left\|\sum_{\mathbf{m}\in \zz^{n+1}}|C_{\mathbf{m}}^s|^q \Big\|\sum_{k\in\zz} | \T_{m\Phi_\nu^{(3),k}}(\Delta_k^{\mathbf m}(f), \widetilde{\Delta}_k^{\mathbf m}(g)) \Big\|_{L^q}^q\right\|_{L^{\frac{p}{q}}}^{(\frac{p}{q})_*}\\
&  \le \sum_{\mathbf{m}\in \zz^{n+1}}|C_{\mathbf{m}}^s|^{q (\frac{p}{q})_*} \left\| \Big\|\sum_{k\in\zz} | \T_{m\Phi_\nu^{(3),k}}(\Delta_k^{\mathbf m}(f), \widetilde{\Delta}_k^{\mathbf m}(g)) \Big\|_{L^q}^q\right\|_{L^{\frac{p}{q}}}^{(\frac{p}{q})_*}\\
&  = \sum_{\mathbf{m}\in \zz^{n+1}}|C_{\mathbf{m}}^s|^{q (\frac{p}{q})_*} \left\| \sum_{k\in\zz} | \T_{m\Phi_\nu^{(3),k}}(\Delta_k^{\mathbf m}(f), \widetilde{\Delta}_k^{\mathbf m}(g)) \right\|_{L^pL^q}^{q(\frac{p}{q})_*}\\
&  \lesssim  \sum_{\mathbf{m}\in \zz^{n+1}}|C_{\mathbf{m}}^s|^{q(\frac{p}{q})_*}  \left\| \(\sum_{k\in\Z}|\Delta_k^{\mathbf m}(f)|^2\)^{1/2}\right\|_{L^{p_1}L^{q_1}}^{q(\frac{p}{q})_*} \!
  \left\|\(\sum_{k\in\Z}|\widetilde{\Delta}_k^{\mathbf m}(g)|^2\)^{1/2}\right\|_{L^{p_2}L^{q_2}}^{q(\frac{p}{q})_*}\\
&  \lesssim  \sum_{\mathbf{m}\in \zz^{n+1}}|C_{\mathbf{m}}^s|^{q(\frac{p}{q})_*} [\ln(1+|\mathbf{m}|)]^{2q(\frac{p}{q})_*} \left\| f\right\|_{L^{p_1}L^{q_1}}^{q(\frac{p}{q})_*}
  \left\| g \right\|_{L^{p_2}L^{q_2}}^{q(\frac{p}{q})_*}\\
&  \lesssim  \left\| f\right\|_{L^{p_1}L^{q_1}}^{q(\frac{p}{q})_*}
  \left\| g \right\|_{L^{p_2}L^{q_2}}^{q(\frac{p}{q})_*},
\end{align*}
again by the hypothesis on $s$ because $q(p/q)_*(n+1+s)$ is either $q(n+1+s)$ or $p(n+1+s)$, which are both bigger than $n+1$.
The result now follows.\end{proof}

We establish first the boundedness of bilinear multipliers of order zero with limited smoothness in mixed Lebesgue spaces following the techniques of \cite{Tomita}, \cite{FT}, and \cite{GMT}.

\begin{theorem}\label{multiplier0}
Let $(n+1)/2<r \le n+1$.
Let $(n+1)/r<p_i,q_i<\infty$, $i=1,2$, $1/p=1/p_1+1/p_2$ and $1/q=1/q_1+1/q_2$.
Suppose that $m$ satisfies \eqref{smooth} in $\rr^{n+1}$ with $\nu=0$.
Then, for all $f,g\in \mathcal{S}(\rr^{n+1})$,
\begin{align*}
 & \|\T_{m}(f,g)\|_{L^{p}L^q}  \lesssim \sup_{j\in \zz}\|m_j\|_{W^{(r,r),2}(\rr^{n+1})}  \|f\|_{L^{p_1}L^{q_1}}\|g\|_{L^{p_2}L^{q_2}}.
\end{align*}
\end{theorem}
\begin{proof}
Following the arguments in, for example,  \cite[(4.1)-(4.13)]{FT}, we can decompose $m$ as the sum of three functions
$$
m(\xi,\eta)=
 m_{1,2}(\xi,\eta) +  m_{2,1}(\xi,\eta) + m_{2,2}(\xi,\eta).
$$
The arguments referred to again work in any dimension and lead to three bilinear multiplier operators satisfying the pointwise estimates
\begin{align*}
&|\T_{m_{2,2}}(f,g)(t,x)|
\lesssim \left( \sup_k\|m_{k}\|_{W^{(r,r),2}}\right) \\
&\,\,\,\,\, \times \left(\sum_{k\in\zz} M(|\tilde{\Delta}_{k})f|^l)(t,x)^{2/l}\right)^{1/2}
\left(\sum_{k\in\zz} M(|\tilde{\Delta}_{k})g|^l)(t,x)^{2/l}\right)^{1/2},
\end{align*}
\begin{align*}
&|\Delta_j\T_{m_{2,1}}(f,g)(t,x)| \\
&\lesssim \left( \sup_k\|m_{k}\|_{W^{(r,r),2}}\right)  \sum_{k=-2}^{2} M(|\tilde{\Delta}_{j+k})f|^l)(t,x)^{1/l} M(|g|^l)(t,x)^{1/l},
\end{align*}
and
\begin{align*}
&|\Delta_j\T_{m_{1,2}}(f,g)(t,x)| \\
&\lesssim \left( \sup_k\|m_{k}\|_{W^{(r,r),2}}\right)  \sum_{k=-2}^{2} M(|\tilde{\Delta}_{j+k})g|^l)(t,x)^{1/l}
M(|f|^l)(t,x)^{1/l}.
\end{align*}
where $\tilde{\Delta}_{j}$ is another Littlewood-Paley operator and $\max\{1,\frac{n+1}{r}\}<l<2$.

Choosing  now $\max\{1,\frac{n+1}{r}\}<l<\min(p_1,p_2,q_1,q_2,2)$ (which is possible by hypothesis),
using H\"older's inequality in mixed Lebesgue spaces, Lemma \ref{maximalpq}, and \eqref{twpq}, we obtain
\begin{align*}
\|\T_{m_{2,2}} (f, g)\|_{L^pL^q}
&\lesssim \sup_k\|m_{k}\|_{W^{(r,r),2}} \|f\|_{L^{p_1}L^{q_1}}\|g\|_{L^{p_2}L^{q_2}}.
\end{align*}

To estimate $\T_{m_{2,1}}$ and $\T_{m_{1,2}}$, however, is where, if either $p\leq 1$ or $q\leq 1$, we need to use the new estimate \eqref{aim10}, which is proved in the Appendix. It follows from such estimate that we can still control the $L^pL^q$ norm by the $H^{p,q}$ one. Therefore,
\begin{align*}
& \|\T_{m_{2,1}} (f, g)\|_{L^pL^q} \lesssim \left\|\left(\sum_j|\Delta_j\T_{m_{2,1}}(f,g)|^2\right)^{1/2}\right\|_{L^pL^q}\\
&\lesssim \sup_k\|m_{k}\|_{W^{(r,r),2}} \bigg\|\Big(\sum_j|M(|\tilde{\Delta}_{j} f_1|^l)^{2/l}\Big)^{1/2}\bigg\|_{L^{p_1}L^{q_1}}\|M(|f_2|^l)^{1/l}\|_{L^{p_2}L^{q_2}}\\
&\lesssim \sup_k\|m_{k}\|_{W^{(r,r),2}} \|f\|_{L^{p_1}L^{q_1}}\|g\|_{L^{p_2}L^{q_2}}.
\end{align*}

Similarly, we can prove
\begin{align*}
\|\T_{m_{1,2}} (f, g)\|_{L^pL^q}
&\lesssim \sup_k\|m_{k}\|_{W^{(r,r),2}} \|f\|_{L^{p_1}L^{q_1}}\|g\|_{L^{p_2}L^{q_2}},
\end{align*}
and the boundedness of $T_m$ follows.
\end{proof}

We finally now extend Theorem \ref{multiplier 1} to the ${L^{p}L^{q}}(\R^{n+1})$.

\begin{theorem}\label{multiplier 2}
Suppose that $m$ satisfies \eqref{smooth} for $0\le\nu<2(n+1)$ and $(n+1)/2<r\le n+1$.
Let $p_i,q_i\in (\frac{n+1}{r},\infty)$, $i=1,2$, $1/p=1/p_1+1/p_2$, and $1/q=1/q_1+1/q_2$.
Then for $s\in 2\nn$ or $s>\max(0,\frac{n+1}{p}-n-1,\frac{n+1}{q}-n-1)$,
$$
  \|D^s\T_{m}(f,g)\|_{L^{p}L^{q}} \lesssim $$$$\sup_{j\in \zz}\|m_j^\nu\|_{W^{(r,r),2}} (\|D^{s-\nu} f\|_{L^{p_1}L^{q_1}}\|g\|_{L^{p_2}L^{q_2} }+\|f\|_{L^{p_1}L^{q_1}}\|D^{s-\nu}g\|_{L^{p_2}L^{q_2}})
$$
for all $f,g\in \mathcal{S}(\rr^{n+1})$.
\end{theorem}
\begin{proof}
There is really not much to prove. We can verify that $m$ satisfies the hypothesis of Theorem \ref{main2} repeating the arguments in the proof of Theorem \ref{multiplier 1}.  That  $\T_{m\Phi_\nu}$ is bounded follows now from Theorem \ref{multiplier0}, given condition {\it (i)} in Theorem \ref{main2}.
To verify condition {\it (ii)} in such theorem, we note that the pointwise estimate \eqref{3.8} still holds in $\R^{n+1}$ and now reads
\begin{equation*}\label{3.81}
\T_{m\Phi_\nu^k}(f,g)(t,x)
 \lesssim \sup_{j\in\zz}\|m_j^\nu \|_{W^{(r,r),2}} (M(|f|^l))^{1/l}(t,x)(M(|g|^l))^{1/l}(t,x)
\end{equation*}
for $l\in (n/r,\min\{2,p_1,p_2,q_1,q_2\})$.
The simple observation that
$$\| |f|^l\|_{L^{p}L^{q}} = \| |f|\|^{1/l}_{L^{p/l}L^{q/l}}.$$
and the rest of the arguments used before, invoking again Lemma \ref{maximalpq}  instead of the Fefferman-Stein inequality, gives the desired result. \end{proof}

The following corollaries are immediate now.

\begin{corollary}
Let $1<p_1,p_2, q_1,q_2<\infty$, $1/p=1/p_1+1/p_2$, and $1/q=1/q_1+1/q_2$. Then, for  $s\in 2\nn$ or $s>\max(0,\frac{n}{p}-n,\frac{n}{q}-n)$ and all $f,g\in \mathcal{S}(\rr^{n+1})$,
$$
  \| D^s(f\cdot g)\|_{{L^{p}L^{q}} } \lesssim \|D^s f\|_{L^{p_1}L^{q_1}} \|g\|_{L^{p_2}L^{q_2}}+\|f\|_{L^{p_1}L^{q_1}}\|D^s g\|_{L^{p_2}L^{q_2}}.
$$
\end{corollary}

\begin{corollary}
Let $0<\nu<2n$ and let $1<p_1,p_2, q_1,q_2<\infty$, $1/p=1/p_1+1/p_2$, and $1/q=1/q_1+1/q_2$.  Then for  for  $s\in 2\nn$ or $s>\max(0,\frac{n}{p}-n,\frac{n}{q}-n)$ and all $f,g\in \mathcal{S}(\rr^{n+1})$,
$$
  \|D^s \I_\nu(f,g)\|_{L^{p}L^{q}} \lesssim \| D^{s-\nu}f\|_{L^{p_1}L^{q_1}}\|g\|_{L^{p_2}L^{q_2}}+\|f\|_{L^{p_1}L^{q_1}}\|D^{s-\nu}g\|_{L^{p_2}L^{q_2}}.
$$
\end{corollary}

\begin{remark}
{\rm We note that the improvement property of $\I_\nu$ in the mixed Lebesgue scale (without derivatives) is trivial. In fact,
$$
  \| \I_\nu(f,g)\|_{L^{r}L^{s}} \lesssim \| f\|_{L^{p_1}L^{q_1}}\|g\|_{L^{p_2}L^{q_2}},
$$
for all $1<p_1,p_2,q_1,q_2<\infty$ satisfying $1/r= 1/p_1+1/p_2 - \nu/(n+1)>0$ and $1/s= 1/q_1+1/q_2 - \nu/(n+1)>0$.
This simply follows from
$$
\I_\nu(f,g)(t,x) \leq I_{\nu/2}|f|(t,x) I_{\nu/2}|g|(t,x),
$$
H\"older's inequality in mixed Lebesgue spaces, and the fact that
$$
I_{\nu/2}: L^{p_0}L^{q_0} \to L^{p_1}L^{q_1}
$$
if   $1/p_0 - 1/p_1 = 1/q_0 -1/q_1 = \nu/2(n+1)$ (see for example Bendek-Panzone \cite{BP} or Moen \cite{Mo2}).}
\end{remark}

\begin{remark}
{\rm Both in the results in this section and the previous one, we could have used conditions of the form
$\sup_{k\in \zz}\|m_k^\nu\|_{W^{(r_1,r_2),2}}<\infty$
for appropriate $r_1 \neq r_2$. We decided to use just $r_1=r_2$ to simplify the presentation. Likewise, we could have used  different  exponents $1<p_3,p_4<\infty$ with $1/p_3+1/p_4=1/p$ (and similarly with $q$) for the second terms on the right-hand side of the estimates in each theorem proved.}
\end{remark}

\begin{remark}
{\rm Under the stronger pointwise smoothness assumption \eqref{verysmooth}, it is possible also to obtain weak-type estimates in the results if $p_1=1$ or $p_2=1$ in the Lebesgue case and also in the outside norm in the mixed Lebesgue case (though still requiring  $q_1, q_2>1$). We refer the interested reader to \cite{GO} to see how the arguments there could be adapted to the ones presented here.}
\end{remark}


\section{Appendix}

The purpose of this Appendix is to provide a proof of the following result which we have used in Section 4.

\begin{theorem}\label{h to l:f}
Let $0<q,p<\infty$. If $f\in H^{p,q}({\R}^{n+1})\cap
L^2({\R}^{n+1})$, then $f\in L^{p}L^{q}({\R}^{n+1})$
and there is a constant $C_{p,q}>0$ independent of the $L^2({\R}^{n+1})$ norm of $f$ such that
\begin{align}\label{aim11}
\|f\|_{L^{p}L^q({\R}^{n+1}) }\le C_{p,q}\|f\|_{H^{p,q}({\R}^{n+1})}.
\end{align}
\end{theorem}

\begin{remark}
{\rm In the applications in Section 4, the multipliers $T_m(f,g)$ are known to be in $L^2$ whenever $f,g\in \S(\rr^{n+1})$,
so Theorem \ref{h to l:f} is always applicable for our arguments in section 4}.
\end{remark}

To prove Theorem \ref{h to l:f}, we need first to adapt some arguments from the works by Frazier-Jawerth \cite{FJ1, FJ2}. The techniques therein for Besov and Triebel-Lizorkin spaces are very powerful and versatile, and they are adaptable to many other situations. In particular, in the works of Han-Lu \cite{HL}, Ding et al \cite{DHLW}, and others, they are adapted to multiparameter Hardy spaces. More recently, and relevant to our needs, some of the same decomposition techniques have been extended to spaces based on mixed norms.  For $L^pL^q$ spaces with $p,q>1$, this was carried out in \cite{TW} while for general Triebel-Lizorkin spaces based on $L^pL^q$ it was done in the work \cite{GJN} already mentioned in Section~2. To avoid too much repetition with  this already existing literature and for the sake of brevity, we will only summarize here some decomposition results in our context. The arguments to establish them, though lengthy, are to some extent routine, or at least expected for those familiar with the Frazier-Jawerth machinery.  In addition, most of them are explicitly performed in the already cited works. In particular, \cite{GJN} conducts a meticulous analysis recasting many of the needed tools. It is important to point out that the crux of the so-called {\it $\varphi$-transform decompositions} of Frazier-Jawerth is to rely on pointwise estimates involving almost orthogonality properties, the Peetre maximal functions, and  the Hardy-Littlewood maximal functions, which hold in any number of dimensions and hence for functions or distributions defined now in $\R^{n+1}$. They are then put together through vector valued estimates involving the Fefferman-Stein result. In the case of mixed Lebesgue space, Lemma~\ref{maximalpq} plays the corresponding role. Given a particular discrete decomposition stated below we do provide a full proof of Thereom~\ref{h to l:f}.  Though our arguments are borrowed in part from the ones in \cite{HL} and \cite{DHLW}, some of which can in turn be traced back in the multiparameter setting to the works of Chang-Fefferman \cite{CF1,CF2}, we face some new technical issues because of the mixed norms. We focus on addressing such issues in detail.

For the a Littlewood-Paley function $\psi$ chosen so that
$$\sum_{j \in\zz}|\widehat \psi(2^{-j} \xi)|^2 =1,  \text { for all } \xi\neq 0 \text{ in }\R^{n+1},$$
Frazier-Jawerth \cite{FJ1,FJ2} showed through a version of the sampling theorem that
\begin{equation}\label{phitransform}
f(t,x)=\sum_{j\in \zz} \, \sum_{\ell(Q)=2^{-j}}  |Q| \, \psi_{2^{-j}}*f(t_Q,x_Q)\, \psi_{2^{-j}}(t-t_Q,x-x_Q) ,
\end{equation}
where for each dyadic cube $Q$ in  $\R^{n+1}$ with side length
$\ell(Q)= 2^{-j}$, $$(t_Q,x_Q)= (2^{-j}k_0,2^{-j}k_n),$$  with $k_0 \in \zz$ and  $k_n \in \zz^n$, is its lower left corner. 
Also, using the notation $$\psi_Q(x,t)= |Q|^{1/2}\psi_{2^{-j}}(t-t_Q,x-x_Q),$$ the reproducing formula \eqref{phitransform} takes the more wavelet-looking form
\begin{equation}\label{normalizedphitransform}
f(t,x)=\sum_Q\<f,\psi_Q\>\psi_Q(t,x),
\end{equation}
 where the sum runs over all dyadic cubes in $\R^{n+1}$. It is known that this wavelet-type decomposition can be used to give discrete characterizations of all function spaces admitting Littlewood-Paley decomposition. Since we have defined $H^{p,q}$ via a Littlewood-Paley square function (quasi-)norm, it is natural that one  also has
\begin{equation*}
\|f\|_{H^{p,q}({\R}^{n+1})}\approx
\left\|\left(\sum_{j\in\zz}\sum_{\ell(Q)=2^{-j}}|\psi_{2^{-j}}*f(t_Q,x_Q)|^2\chi_Q(t,x)\right)^{1/2}\right\|_{L^pL^q(\R^{n+1)}}
\end{equation*}
as proved in \cite{TW} for $p,q>1$ and for all exponents in \cite{GJN}. Incidentally, as also proved in \cite{GJN}, this characterization and related result can be used as in the case of spaces based on $L^p$ to prove  that the definition of $H^{p,q}$ does not depend on the choice of function $\psi$.

It is also possible (and sometimes convenient) to obtain a version of the discrete Calder\'on reproducing formula \eqref{phitransform} using two generating family of functions, one of which actually has compact support (and an arbitrarily large, but finite, number of vanishing moments), but of course we can no longer have the samples of the functions
$ \psi_{2^{-j}}*f$ as coefficients.  Such discrete formula is referred to as the {\it generalized $\varphi$-transform}. We state below the equivalent formulation for spaces based on mixed norms. We skip the details of the proof, but once again we refer the reader to \cite{HL, TW, GJN} for the tools to apply the Frazier-Jawerth blueprints  in the mixed-norm context.

Let $\phi\in\S({\R}^{n+1})$ be supported on $B(0,2)$ satisfy
\begin{equation*}
\int_{{{\R}}^{n+1}}\phi (t,x)t^{\alpha}x^{\beta} dx\,dt=0
\end{equation*}
for $\alpha\in\N_0$ and $\beta\in\N_0^n$ with $|\alpha|+|\beta|\le M$, where $M$ is a fixed positive integer.  (Such a function exists, and a construction is given in \cite[p. 783]{FJ1}.)  Further assume that $\widehat\phi(\xi)\geq c>0$ for $\frac{1}{2}\leq|\xi|\leq2$, and define now
$$
\sigma^Q(t,x)=|Q|^{1/2} \phi_{2^{-(j+N)}}(t-t_Q,x-x_Q),
$$
for each dyadic cube $Q\subset\R^{n+1}$ of side length $2^{-j}$ and with lower left corner $(t_Q,x_Q)$, where $N<0$ is some fixed integer.  The work of Frazier-Jawerth \cite[Theorem 4.2]{FJ2} can be modified to construct, for $|N|$ large enough, a family of functions $\tau^Q$, indexed by dyadic cubes $Q$, such that 
\begin{equation}
f(t,x)=\sum_Q\<f,\tau^Q\>\sigma^Q(t,x),\label{gphitransform}
\end{equation}
where again the sum in $Q$ is over all dyadic cubes in $\R^{n+1}$.  The convergence of the formula in \eqref{gphitransform}, like the one in \eqref{phitransform}, holds in a very general sense but certainly in $L^2(\R^{n+1})$.  Moreover,
\begin{equation}
\|f\|_{H^{p,q}(\R^{n+1})}\approx\|\widetilde g(f)\|_{L^pL^q},\label{LittlewoodPaley}
\end{equation}
where to simplify notation we define $\widetilde g(f)$ to be the discrete Littlewood-Paley square function given by
$$
\widetilde{g}(f)(t,x):=\(\sum_{Q}(|Q|^{-1/2}|\<f,\tau^Q\>|)^2\chi_Q(t,x)\)^{1/2}.
$$

For clarification purposes, we point out that the numbers $M$ and $|N|$ which are needed to be taken sufficiently large based on $p$, $q$ and the dimension, will play no role in our arguments.  We also note that while the functions $\sigma^Q$ are the translations and dilations of a single function with compact support, the functions $\tau^Q$ are not.  We will not need to know  the functions $\tau^Q$ explicitly for our proof. See again \cite{FJ2} for their properties.

One last technical detail that will be convenient to us (see the proof of Theorem \ref{h to l:f} below) is to assume that  the function $\phi$ used in \eqref{phitransform} is of the form $\phi=\varphi*\varphi$ where $\varphi$ is a real-valued, radial function $\varphi\in\S(\R^{n+1})$, also satisfies the vanishing moments conditions and is supported on $B(0,1)$.  This can be easily done by following the construction in \cite{FJ1} mentioned earlier.

With these technical issues and preliminary facts about $H^{p,q}(\R^{n+1})$, we can prove now Theorem~\ref{h to l:f}.
\begin{proof}
The case $p>1$ and $q>1$ is known (cf. \cite{TW}).
We then consider two cases separately.

\noindent\textbf{Case 1: } $0<q\le 1$ and $0<p<\infty$.

Let $f\in L^2({\R}^{n+1})\cap H^{p,q}({\R}^{n+1})$. For every $i\in\zz,$ set
$$\Omega_i=\{(t,x)\in {\R}^{n+1}: \widetilde{g}(f)(t,x)>2^i\}$$ and
$$
B_i=\{(j,Q): j\in{\zz}, Q\in\Q_{j}, |Q\cap \Omega_i|>(1/2)|Q|, |Q\cap
\Omega_{i+1}|\le (1/2)|Q|\},
$$
where $\Q_j$ is the collection of all dyadic cubes with side length $2^{-j}$.  For $\ f\in L^2({\R}^{n+1})\cap H^{p,q}({\R}^{n+1})$, we rewrite \eqref{gphitransform} as
\begin{equation}\label{C1}
f=\sum_{j\in\Z}\sum_{Q\in\Q_j}|Q|^{1/2}\<f,\tau^Q\>\phi_{2^{-(j+N)}}(\cdot-t_Q,\cdot-x_Q)
\end{equation}
Select $\tilde p, \tilde q\ge 2$ such that
$$
\frac{\tilde p}{p}=\frac{\tilde q}{q}.
$$
Let $1/r=1/q-1/\q$
and let $s$ be the dual index of $r$ if $r>1$ and $s=2$ otherwise. Note that $s\ge 2$ since if $r>1$,
$1/r=1/q-1/\q\ge 1-1/2=1/2.$

Define
$$
a_i=2^{-i(1-\frac{q}{\widetilde{q}})}\sum_{(j,Q)\in B_i}|Q|^{1/2}\<f,\tau^Q\>\phi_{2^{-(j+N)}}(\cdot-t_Q,\cdot-x_Q).$$
We note that for any $i\in\Z$, $a_i$ is supported in
$$
\widetilde{\Omega}_i:=\{(t,x):{{M}}(\chi_{\Omega_i})(t,x)>2^{-(1+(n+1)(3-N))}\}.
$$
In fact,  $\phi_{2^{-(j+N)}}(t-t_Q,x-x_Q)$ is supported in $2^{3-N}Q$,  and $|Q\cap \Omega_i|>|Q|/2$ for $(j,Q)\in B_i$. Hence for
$$(t,x)\in \supp(\phi_{2^{-(j+N)}}(\cdot-t_Q,\cdot-x_Q))\subset 2^{3-N}Q,$$
we have
\begin{align*}
M(\chi_{\Omega_i})(t,x)&\geq\frac{2^{-(n+1)(3-N)}}{|Q|}\int_{2^{3-N}Q}\chi_{\Omega_i}dt\,dx\geq\frac{2^{-(n+1)(3-N)}}{|Q|}\int_{Q}\chi_{\Omega_i}dt\,dx\\
&=2^{-(n+1)(3-N)}\frac{|Q\cap \Omega_i|}{|Q|}>2^{-(1+(n+1)(3-N))}.
\end{align*}
Thus, by \eqref{C1},
$$f=\sum_{i} 2^{i(1-q/\q)}\chi_{\widetilde{\Omega}_i} a_i. $$
Since $s\geq 2$, it follows that
\begin{align*}
\sum_ia_ib_i & \le \left(\sum_i|b_i|^r\right)^{1/r}\left(\sum_i|a_i|^s\right)^{1/s}\\
& \le \left(\sum_i|b_i|^r\right)^{1/r}\left(\sum_i|a_i|^2\right)^{1/2}.
\end{align*}
Therefore
\begin{align}\label{eq 66662}
\begin{split}
\|f\|_{\lpq}&=\Big\|\sum_{i} 2^{i(1-q/\q)}\chi_{\widetilde{\Omega}_i} a_i\Big\|_{\lpq({\R}^{n+m}) }\\
& \le \left\|\left(\sum_i 2^{iq}\chi_{\widetilde{\Omega}_i}\right)^{1/r} \left(\sum_i|a_i|^2\right)^{1/2}\right\|_{\lpq}.
\end{split}
\end{align}
Also, since  $1/q=1/r+1/\q$,  we have  $$1/q =p/(q\p) + 1/r$$  or  $$1/p= q/(rp )+ 1/\p.$$ Hence, by H\"{o}lder's inequality,
\begin{equation*}
  \|FG\|_{L^pL^q}\le \left[\int \left(\int|F|^{r}\right)^{p/q}\right]^{1/p-1/\p}\|G\|_{L^{\p}L^{\q}}.
\end{equation*}
Applying this with $F=(\sum_i 2^{iq}\chi_{\widetilde{\Omega}_i})^{1/r}$ and $G=(\sum_i|a_i|^2 )^{1/2}$, we obtain from \eqref{eq 66662}
\begin{align*}
\|f\|_{\lpq}
&\lesssim \left(\int_{\rr}\left(\int_{\rr^n}\sum_i 2^{iq}\chi_{\widetilde{\Omega}_i} dx\right)^{\frac{p}{q}}dt\right)^{1/p-1/\tilde{p}}\Big\| \left(\sum_i|a_i|^2 \right)^{1/2}\Big\|_{L^{\tilde p}L^{\tilde q} }\\
& = I\times J.
\end{align*}
Let us first estimate $I.$ Pick  $u>\max (\frac{q}{p},1).$ By the definition of $\widetilde\Omega_i$ and Lemma~\ref{maximalpq}, it follows that
\begin{align*}\label{Obser2}
 &   \left(\int_{\rr}\left(\int_{\rr^n}   \sum_{i\in{\zz}} 2^{qi}\chi_{\widetilde{\Omega}_i}(t,x)dx\right)^{\frac{p}{q}}dt\right)^{1/p}  \\
 &   \quad \quad \quad \quad  \lesssim \left(\int_{\rr}\left(\int_{\rr^n}\sum_{i\in{\zz}} ( M(2^{qi/u}\chi_{\Omega_i})(t,x) )^{u}dx\right)^{\frac{p}{q}}dt\right)^{1/p} \\
   & \quad \quad \quad \quad  =\|\{M(2^{qi/u}\chi_{\Omega_i})\}\|_{L^{up/q}_tL^{u}_x(\ell^{u})}^{u/q} \\
   & \quad \quad \quad \quad  \lesssim \|\{2^{qi/u}\chi_{\Omega_i}\}\|_{L^{up/q}_tL^{u}_x(\ell^{u})}^{u/q}\\
   & \quad \quad \quad \quad \le  \left(\int_{\rr}\left(\int_{\rr^n}\sum_{i=-\infty}^{\lfloor\log_2\widetilde{g}(f)(t,x)\rfloor+1}2^{qi}dx\right)^{\frac{p}{q}}dt\right)^{1/p}\\
   & \quad \quad \quad \quad  \lesssim \left(\int_{\rr}\left(\int_{\rr^n}\widetilde{g}(f)(t,x)^{q}dx\right)^{\frac{p}{q}}dt\right)^{1/p}\\
   & \quad \quad \quad \quad  \approx \|f\|_{H^{p,q}},
\end{align*}
and consequently
$I\lesssim \|f\|_{H^{p,q}}^{1-\frac{p}{\p}}.$

Next, we  show that
\begin{equation}\label{6new}
\left\|\left (\sum_i|a_i|^2 \right )^{1/2}\right\|_{L^{\tilde p}L^{\tilde q} }\lesssim \|f\|_{H^{p,q}}^{\frac{p}{\p}}.
\end{equation}
For $\{\zeta_i\}\in L^{\tilde{p}'}\!L^{\tilde{q}'}(\ell^2)$ with
$\|\{\zeta_i\}\|_{ L^{\tilde{p}'}\!L^{\tilde{q}'}(\ell^2)}\le 1$, and denoting by  $\langle \! \langle\cdot,\cdot\rangle \!\rangle$ the pairing between $L^{\tilde p}L^{\tilde q}(\ell^2)$ and $L^{\tilde p'}L^{\tilde q'}(\ell^2)$, we get  %
\begin{align*}
& \left| \left\langle \! \! \!  \left\langle  \left\{2^{-i(1-\frac{q}{\widetilde{q}})}\sum_{(j,Q)\in B_i}|Q|^{1/2}\<f,\tau^Q\>\phi_{2^{-(j+N)}}(\cdot-t_Q,\cdot-x_Q) \right\}, \left\{ \zeta_i (\cdot,\cdot)\right\} \right \rangle \! \! \! \right\rangle\right |  \\
& = \left| \sum_i2^{-i(1-\frac{q}{\widetilde{q}})}\sum_{(j,Q)\in B_i }\int_\R \int_{\R^{n}} |Q|^{-1/2}\<f,\tau^Q\> \right. \\
&  \quad \quad \quad  \quad \quad \quad  \quad \quad \quad  \quad \quad \quad \left. \phantom{\sum_i} \phantom{\int}  \times  \phi_{2^{-(j+N)}}*\zeta_i(t_Q,x_Q)\chi_Q(t,x) \, dxdt \right |\\
&\le \left\| \left( \sum_i2^{-2i(1-\frac{q}{\widetilde{q}})}\sum_{(j,Q)\in B_i }( |Q|^{-1/2}|\<f,\tau^Q\>|)^{2}\chi_Q \right)^{\frac{1}{2}}\right\|_{L^{\tilde p}L^{\tilde q}}  \\
& \quad \quad \quad  \quad \quad \quad  \quad \quad \quad
          \times  \left\| \left( \sum_i \sum_{(j,Q)\in B_i }|\phi_{2^{-(j+N)}}*\zeta_i(t_Q,x_Q)|^{2}\chi_Q \right) ^{\frac{1}{2}}
          \right\|_{L^{\tilde p'}L^{\tilde q'}} \\
&\lesssim \left\| \left( \sum_i2^{-2i(1-\frac{q}{\widetilde{q}})}\sum_{(j,Q)\in B_i }( |Q|^{-1/2}|\<f,\tau^Q\>|)^{2}\chi_Q \right)^{\frac{1}{2}}\right\|_{L^{\tilde p}L^{\tilde q}}.
\end{align*}
The last inequality above follows by the selection of $\phi=\varphi*\varphi$, the trivial estimate
$$|\varphi_{2^{-(j+N)}}*F(t_Q,x_Q)| \chi_Q(x,t)\less M(F)(t,x) \chi_Q(x,t),$$
 Lemma~\ref{maximalpq}, and a Littlewood-Paley estimate for mixed Lebesgue spaces, as the computations below show,
\begin{align*}
& \left\| \left( \sum_i \sum_{(j,Q)\in B_i }|\phi_{2^{-(j+N)}}*\zeta_i(t_Q,x_Q)|^{2}\chi_Q \right) ^{\frac{1}{2}}
          \right\|_{L^{\tilde p'}L^{\tilde q'}}  \\
  &  \lesssim      \left\| \left( \sum_i \sum_{(j,Q)\in B_i }|M(\varphi_{2^{-(j+N)}}*\zeta_i)|^{2}\chi_Q \right) ^{\frac{1}{2}}
          \right\|_{L^{\tilde p'}L^{\tilde q'}}  \\
&    \lesssim       \left\| \left( \sum_{i} \sum_{j\in \zz} |M(\varphi_{2^{-j}}*\zeta_i)|^{2} \right) ^{\frac{1}{2}}
          \right\|_{L^{\tilde p'}L^{\tilde q'}} \\
&    \lesssim       \left\| \left( \sum_{i} \sum_{j\in \zz} |\varphi_{2^{-j}}*\zeta_i|^{2} \right) ^{\frac{1}{2}}
          \right\|_{L^{\tilde p'}L^{\tilde q'}} \\
&    \lesssim       \left\| \left( \sum_i |\zeta_i|^{2} \right) ^{\frac{1}{2}}
          \right\|_{L^{\tilde p'}L^{\tilde q'}}\\
     &    \lesssim  1     .
\end{align*}
The version of the vector valued Littlewood-Paley estimate applied above in the second to last inequality can be found for Lebesgue spaces, for example, in the book of Grafakos \cite[Theorem 5.1.4]{G}. Such version also easily extends to the mixed Lebesgue space setting by the results in \cite{BCP}.

By duality, we then have
\begin{align}\label{6.14}
& \left\|\left (\sum_i|a_i|^2 \right )^{1/2}\right\|_{L^{\tilde p}L^{\tilde q} } \nonumber\\
& \quad \quad  \lesssim \left\|\left( \sum_i2^{-2i(1-\frac{q}{\widetilde{q}})}\sum_{(j,Q)\in B_i }( |Q|^{-1/2}|\<f,\tau^Q\>|)^{2}\chi_Q\right)^{\frac{1}{2}}\right\|_{L^{\tilde p}L^{\tilde q}}.
\end{align}
For any $i\in\Z$ and $(j,Q)\in B_i$ and $(t,x)\in Q$, it also follows that $Q\subset \widetilde\Omega_i$ and that 
\begin{align*}
M(\chi_{Q\cap\widetilde\Omega_i\setminus\Omega_i})(t,x)\ge \frac{|Q\backslash \Omega_i|}{|Q|}>\frac12.
\end{align*}

Applying again Lemma~\ref{maximalpq}, the definition of $B_i$, and the above estimates  in \eqref{6.14}, we have
\begin{align*}
&\left \| \left( \sum_i2^{-2i(1-\frac{q}{\widetilde{q}})}\sum_{(j,Q)\in B_i }( |Q|^{-1/2}|\<f,\tau^Q\>|)^{2}\chi_Q \right)^{\frac{1}{2}}
\right \|_{L^{\tilde p}L^{\tilde q}}^{\tilde p}\\
&\lesssim\left \| \left( \sum_i2^{-2i(1-\frac{q}{\widetilde{q}})} \!\sum_{(j,Q)\in B_i }( |Q|^{-1/2}|\<f,\tau^Q\>|)^{2} M (\chi_{Q\cap\widetilde\Omega_i\setminus\Omega_i})^2 \right)^{\frac{1}{2}} \right \|_{L^{\tilde p}L^{\tilde q}}^{\tilde p}\\
&\lesssim\left \| \left( \sum_i2^{-2i(1-\frac{q}{\widetilde{q}})}\sum_{(j,Q)\in B_i }( |Q|^{-1/2}|\<f,\tau^Q\>|)^{2} \chi_{Q\cap\widetilde\Omega_i\setminus\Omega_i} \right)^{\frac{1}{2}}  \right \|_{L^{\tilde p}L^{\tilde q}}^{\tilde p}\\
&\lesssim \int_{\R} \left( \int_{\rr^n} \left( \sum_i 2^{-2i(1-\frac{q}{\widetilde{q}})} 2^{2i}
\chi_{\widetilde{\Omega}_i \setminus \Omega_i} (t,x)
 \phantom{ \left( \left. \sum_{(j,Q)\in B_i } |\phi_{2^{-j}} \right. dx \right. ^{\frac{\tilde{p}}{\tilde{q}}}\,dt }
 \right. \right.    \\
&\quad \quad\quad \quad \quad \quad \quad \quad    \left.  \left.
       \times \sum_{(j,Q)\in B_i } ( |Q|^{-1/2}|\<f,\tau^Q\>|)^{2} \chi_{Q}(t,x)  \right) ^{\frac{\tilde{q}}{2}}dx \right) ^{\frac{\tilde{p}}{\tilde{q}}}\,dt   \\
&\lesssim\int_{{\R}^{m}}\left(\int_{\rr^n}\left( \sum_i2^{-2i(1-\frac{q}{\widetilde{q}})}2^{2i}
\chi_{\widetilde{\Omega}_i}(t,x) \right)^{\frac{\tilde{q}}{2}}dx\right)^{\frac{\tilde{p}}{\tilde{q}}}\,dt\\
&\lesssim\int_{{\R}^{m}}\left(\int_{\rr^n}\sum_i2^{iq}
\chi_{\widetilde{\Omega}_i}(t,x)dx\right)^{\frac{p}{q}}\,dt\\
&\lesssim  \|f\|_{H^{p,q}}^{p},
\end{align*}
where we have used that
$$
\sum_{(j,Q)\in B_i } ( |Q|^{-1/2}|\<f,\tau^Q\>|)^{2} \chi_{Q}(t,x) \lesssim \tilde g(f)(t,x)^2 \lesssim 2^{2i},
$$
and the fact that $q>2$ and $p/q=\p/\q$. This gives \eqref{6new} and hence Theorem \ref{h to l:f} is verified if $0<q\leq 1$.

\bigskip

\noindent\textbf{Case 2:} $0<p\le 1<q<\infty.$

The proof for this case is simple and closer to the Lebesgue space case given in \cite{HL};
we include the details for the reader's convenience.
We set now
$$\Omega'_i=\{t\in {\R}: \|\widetilde{g}(f)(t,\cdot)\|_{L^q_x(\rr^n)}>2^i\}$$
and
$$
B'_i=\{(j,Q): j\in{\zz}, Q\in\Q_{j}, |Q'\cap \Omega_i'|>(1/2)|Q'|, |Q'\cap
\Omega_{i+1}'|\le (1/2)|Q'|\},
$$
where
$Q=Q' \times Q''\subset \R \times  \rr^n$.

We want to show first that
\begin{equation}\label{claim3}
 \left \|\sum_{(j,Q)\in B_i'}|Q|^{1/2}\<f,\tau^Q\>\phi_{2^{-(j+N)}}(\cdot-t_Q,\cdot-x_Q) \right \|_{L^pL^q}^p\lesssim 2^{ip}|\Omega_i'|.
\end{equation}

Let $M_1$ denote the maximal function on $\R$. We note that if $(j,Q)\in B_i'$, then $\phi_{2^{-(j+N)}}(t-t_Q, x-x_Q)$, as a function of $t$, is supported in
$$
\widetilde{\Omega}_i':=\{t:{{M}_1}(\chi_{\Omega_i'})(t)>2^{-(4-N)}\}
$$
uniformly in $x\in\R^n$.  That is,
$$
\bigcup_{x\in\R^n}\supp(\phi_{2^{-(j+N)}}(\cdot-t_Q,x-x_Q))\subset\widetilde{\Omega}_i',
$$
and so 
$$\left \|\sum_{(j,Q)\in B_i'}|Q|^{1/2}\<f,\tau^Q\>\phi_{2^{-(j+N)}}(t-t_Q,\cdot-x_Q)\right \|_{L^q_x}$$
is supported in $\widetilde \Omega_i$ as a function of $t$.

By H\"{o}lder's inequality,
\begin{align}\label{6.12}
&\left \| \sum_{(j,Q)\in B_i'}|Q|^{1/2}\<f,\tau^Q\>\phi_{2^{-(j+N)}}(\cdot-t_Q,\cdot-x_Q)\right \|_{\lpq (\R^{n+1})} \nonumber\\
& \lesssim|\widetilde{\Omega}_i'|^{1/p-1/q} \left\| \sum_{(j,Q)\in B_i'}|Q|^{1/2}\<f,\tau^Q\>\phi_{2^{-(j+N)}}(\cdot-t_Q,\cdot-x_Q)\right\|_{L^{q}(\R^{n+1})}.
\end{align}
For $\zeta\in L^{q'}(\R^{n+1})$ with $\|\zeta\|_{ L^{q'}}\le 1$ and the usual duality pairing for $L^q$ spaces in $\R^{n+1}$, we have
\begin{align}\label{6.13}
& \left|\Big\langle\sum_{(j,Q)\in B_i}|Q|^{1/2}\<f,\tau^Q\>\phi_{2^{-(j+N)}}(\cdot-t_Q,\cdot-x_Q),\zeta\Big\rangle\right|  \nonumber\\
& =     \left|\sum_{(j,Q)\in B_i' }\int_{\R^{n+1}}|Q|^{-1/2}\<f,\tau^Q\>\phi_{2^{-(j+N)}}*\zeta(t_Q,x_Q)\chi_Q(t,x) dxdt\right|    \nonumber \\
& \le \left\|\left(\sum_{(j,Q)\in B_i' }(|Q|^{-1/2}|\<f,\tau^Q\>|)^{2}\chi_Q\right)^{\frac{1}{2}}\right\|_{L^{q}({\R}^{n+1})}  \nonumber\\
&\qquad \qquad \qquad \qquad
               \times     \left\|\left(\sum_{(j,Q)\in B_i' }|\phi_{2^{-(j+N)}}*\zeta(t_Q,x_Q)|^{2}\chi_Q\right)^{\frac{1}{2}}
\right\|_{L^{q'}({\R}^{n+1})}   \nonumber \\
& \lesssim \left\|\left(\sum_{(j,Q)\in B_i' }(|Q|^{-1/2}|\<f,\tau^Q\>|)^{2}\chi_Q\right)^{\frac{1}{2}}\right\|_{L^{q}}.
\end{align}
The last line here uses  similar computations to the ones in Case 1, which are as follows,
\begin{align*}
& \left\|\left(\sum_{(j,Q)\in B_i' }|\phi_{2^{-(j+N)}}*\zeta(t_Q,x_Q)|^{2}\chi_Q\right)^{\frac{1}{2}}
\right\|_{L^{q'}({\R}^{n+1})}  \\
&  \lesssim \left\|\left(\sum_{(j,Q)\in B_i' }\(M(\varphi_{2^{-(j+N)}}*\zeta)\)^{2}\chi_Q\right)^{\frac{1}{2}}
\right\|_{L^{q'}({\R}^{n+1})}  \\
&   \lesssim \left\|\left(\sum_{j\in \zz}\(M(\varphi_{2^{-j}}*\zeta)\)^{2}\right)^{\frac{1}{2}}
\right\|_{L^{q'}({\R}^{n+1})}  \\
&  \lesssim \left\|\left(\sum_{j\in \zz}|\varphi_{2^{-j}}*\zeta|^{2}\right)^{\frac{1}{2}}
\right\|_{L^{q'}({\R}^{n+1})}  \\
& \lesssim \left\| \zeta \right\|_{L^{q'}({\R}^{n+1})}.
\end{align*}

For any $(j,Q)\in B_i'$ it follows that $Q'\subset \widetilde\Omega_i'$, and hence for $(t,x)\in Q=Q'\times Q''$,
\begin{align*}
 M(\chi_{[Q'\cap\widetilde\Omega_i'\setminus\Omega_i']\times Q''})(t,x)&\ge \frac{|(Q'\setminus\Omega_i')\times Q''|}{|Q'\times Q''|}=\frac{|Q'\backslash \Omega_i'|}{|Q'|}\ge \frac12.
\end{align*}
By the weak-type boundedness of the maximal operator, $|\widetilde\Omega_i'|\less|\Omega_i'|$.  So combining the estimate above with \eqref{6.12} and \eqref{6.13} gives \eqref{claim3} by the following argument,
\begin{align*}
&\left\|\left(\sum_{(j,Q)\in B_i' }(|Q|^{-1/2}|\<f,\tau^Q\>|)^2\chi_Q\right)^{\frac{1}{2}}\right\|_{L^{q}}^q  \\
& \lesssim\int_{\R}\int_{\rr^n}\left(\sum_{(j,Q)\in B_i'}(|Q|^{-1/2}|\<f,\tau^Q\>|)^2 \left( M (\chi_{[Q'\cap\widetilde\Omega_i'\setminus\Omega_i']\times Q''})(t,x)\right)^2 \right)^{\frac{q}{2}}\!dxdt  \\
& \lesssim\int_{\R}\int_{\rr^n}\left(\sum_{(j,Q)\in B_i'}(|Q|^{-1/2}|\<f,\tau^Q\>|)^2 \chi_{[Q'\cap\widetilde\Omega_i'\setminus\Omega_i']\times Q''}(t,x)\right)^{\frac{q}{2}}\! dxdt  \\
& = \int_{{\R}}\left(\int_{\rr^n}\left(\sum_{(j,Q)\in
B_i'}(|Q|^{-1/2}|\<f,\tau^Q\>|)^2\chi_{Q}(t,x)\right)^{\frac{q}{2}}dx\right)\chi_{\widetilde\Omega_i'\setminus\Omega_i'}(t)\,dt \\
&  \lesssim 2^{iq} |\widetilde{\Omega}_i'|.
\end{align*}


Finally,
\begin{align*}
  \|f\|_{\lpq}^p & \le  \sum_i\left\|\sum_{(j,Q)\in B_i'}|Q|^{1/2}\<f,\tau^Q\>\phi_{2^{-(j+N)}}(\cdot-t_Q, \cdot-x_Q)\right\|_{L^pL^q}^p\\
  &\lesssim \sum_i 2^{ip}|\Omega_i'| \lesssim \int_{\rr}\sum_{i=-\infty}^{\lfloor \log_2\|\widetilde g(f)\|_{L_x^{q}}\rfloor+1}2^{ip}dt.\\
   & \lesssim \left\|\|\widetilde{g}(f)\|_{L^q_x}\right\|_{L^p_t}^p \approx \|f\|_{H^{p,q}}^p,
\end{align*}
and  Theorem \ref{h to l:f} follows in this case too.
\end{proof}

\end{document}